\documentclass[11pt,a4paper]{article}
\usepackage{graphicx}
\usepackage{color,a4wide}
\usepackage{amsmath,amsthm}
\usepackage{amssymb,enumerate}
\theoremstyle{plain}
\newtheorem{thm}{Theorem}[section]
\newtheorem{lem}[thm]{Lemma}
\newtheorem{lemma}[thm]{Lemma}

\theoremstyle{definition}

\theoremstyle{remark}
\newtheorem*{rem}{Remark}

\newcommand{\e}{\nu}
\newcommand{\rhot}{\widetilde{\rho}}
\newcommand{\rhob}{\overline{\rho}}
\newcommand{\R}{\mathbb{R}}

\begin{document}

\title{Stationary States and Asymptotic Behaviour of Aggregation Models with Nonlinear Local Repulsion}
\author{Martin Burger\thanks{Institut f\"{u}r Numerische und Angewandte Mathematik, Westf\"{a}lische Wilhelms-Universit\"{a}t M\"{u}nster,
Einsteinstr. 62, 48149 M\"{u}nster, Germany. Email: martin.burger@wwu.de}
\and Razvan Fetecau\thanks{Department of Mathematics, Simon Fraser University, 8888 University Dr., Burnaby, BC V5A 1S6, Canada. Email: van@math.sfu.ca}
\and Yanghong Huang\thanks{Department of Mathematics
Imperial College London, London SW7 2AZ, United Kingdom. Email: yanghong.huang@imperial.ac.uk}  }

\maketitle

\begin{abstract} 
We consider a continuum aggregation model with 
nonlinear local repulsion given by a degenerate power-law diffusion with general exponent. The steady states and their properties in one dimension are studied both  analytically and numerically, suggesting that the quadratic diffusion is a critical case. The focus is on finite-size, monotone and compactly supported equilibria. We also investigate numerically the long time asymptotics of the model by simulations of the evolution equation. Issues such as metastability and local/ global stability are studied in connection to the gradient flow formulation of the model.

\end{abstract}


\section{Introduction}
The derivation and analysis of mathematical models for collective behaviour of cells, animals, or humans have been receiving increasing attention in recent years. In particular, a variety of continuum models based on evolution equations for population densities has been derived and used to describe biological aggregations  such as flocks and swarms \cite{Grunbaum:msag,MR1698215,MR2117406,TBL}. A typical aspect of these models is the competition of social interactions (repulsion and attraction) between group individuals, which is also the focus of current research.

In this paper we add novel results on the study of a canonical model for the competition of attraction and repulsion, namely the following one-dimensional aggregation equation for the population density $\rho$:
\begin{equation}
\label{eq:aggeq}
 \partial_t \rho +  \partial_x( \rho \partial_x (G \ast \rho))= \nu \partial_x(\rho \partial_x \rho^{m-1}).
\end{equation}
Here, $G$ is an attractive interaction potential (to be detailed in Section~\ref{sec:defG}), 
 $\nu>0$ is a diffusion coefficient  and $m>1$ is a real exponent. Equation \eqref{eq:aggeq} falls into the general class of aggregation equations with degenerate diffusion in arbitrary dimension $n$:
\begin{equation}
\label{eq:main_evo}
    \partial_t \rho +  \nabla \cdot \big( \rho \nabla (G \ast \rho)\big)= 
\nabla \cdot \big(\rho \nabla f(\rho)\big),
\end{equation}
which has been widely studied in applications such as biological swarms \cite{Burger:Capasso,Burger:DiFrancesco,TBL} or chemotaxis \cite{BeRoBe2011,BlCaLa09}.

The left-hand-side of  \eqref{eq:main_evo} represents the
active transport of the density $\rho$ associated to a non-local velocity 
field $\mathbf{v} = \nabla (G \ast \rho)$. The potential $G$ is assumed 
to incorporate only {\em attractive} interactions among individuals of the 
group, while repulsive (anti-crowding) interactions are accounted for by the nonlinear diffusion in the right-hand-side.  Alternatively, by transferring the nonlinear diffusion to the left-hand-side, one can regard equation \eqref{eq:main_evo} as active transport of $\rho$ that corresponds to a velocity field that has a non-local attractive component, $\nabla (G \ast \rho)$, and a {\em local} repulsion or dispersal part, $- \nabla f(\rho)$. Nonlinear diffusion terms have been suggested in several instances for dispersal and repulsion (cf. e.g. \cite{MR736508} and \cite{MR1698215,TBL} in the above context). From a microscopic point of view, the case $m=2$ ($f(\rho)=\nu \rho$) can be easily justified, by taking the repulsive force modelled by a potential, similar to the aggregative one, and then performing a scaling limit as the interaction range approaches zero (cf. \cite{MR2117406}). In Section 2 we present a microscopic derivation of the model with arbitrary $f$ based on nearest-neighbour interactions for 
the repulsion, which provides a unified interpretation of the nonlinear diffusion in terms of local forces. Models of type \eqref{eq:aggeq} also appear in earlier works on population dynamics \cite{Ikeda1985, IkedaNagai1987,NagaiMimura1983}, but the potentials considered there have different properties than the ones considered in this article.

Regardless of the interpretation, there is a delicate balance between attractive and dispersal effects which results in very interesting (and biologically relevant) dynamics and long-time behaviour of solutions to \eqref{eq:main_evo}.
Well-posedness of solutions to  \eqref{eq:main_evo} has been studied intensively, we refer for instance to  \cite{BertozziSlepcev10,Burger:Capasso,Burger:DiFrancesco}. In particular, the analysis in  \cite{Burger:DiFrancesco} takes advantage of the formulation of these models in terms of gradient flows on spaces of probability measures equipped with the Wasserstein metric (cf. \cite{ambrosiogiglisavare}). Also, a wide literature exists  in relation to the Keller-Segel model for chemotaxis (see \cite{BeRoBe2011,BlCaLa09} and references therein). As pointed out in such works, existence theory is more delicate in the presence of singular kernels, where finite time blow-up of solutions is possible \cite{BeRoBe2011,KozonoSugiyama08}. 
 
Of central role in studies of model \eqref{eq:main_evo}, and also particularly relevant to the present research, is the gradient flow formulation~\cite{ambrosiogiglisavare} of the equation with respect to the energy 
\begin{equation}
    E[\rho]:=  \int_{\mathbb{R}^n} F\big(\rho(x)\big) dx 
- \frac{1}{2}\int_{\mathbb{R}^n} \int_{\mathbb{R}^n}  G(x-y) \rho(y) \rho(x) dy dx,
\label{def:entropyfunctional}
\end{equation}
where $F'(\rho) = f(\rho)$. The special case \eqref{eq:aggeq} corresponds to power-law diffusion
\begin{equation}
\label{eqn:powerfF}
f(\rho)=\nu\rho^{m-1}, \quad F(\rho) = \frac{\nu}{m} \rho^m.
\end{equation}
Stationary states of \eqref{eq:main_evo} are critical points of the energy \eqref{def:entropyfunctional}. A recent work by Bedrossian \cite{Bedrossian} investigates the existence of global minimizers of \eqref{def:entropyfunctional} using calculus of variations techniques. In particular,  the existence of a radially symmetric and non-increasing global minimizer can be inferred for power-law $F(\rho)$ in \eqref{eqn:powerfF} with $m > 2$. The case $m=2$ is critical and yields a global minimizer only for small enough diffusion coefficient $\nu$, where the threshold value for $\nu$ is shown to be $\|G\|_{L^1}$ by Burger {\em et al.} in \cite{BurgerDiFrancescoFranek}. Energy considerations have also been employed in  \cite{Burger:DiFrancesco} to study the large time behaviour of solutions to  \eqref{eq:main_evo} in one dimension.


This paper considers the power-law diffusion \eqref{eqn:powerfF},
though the results are expected to be true for any strictly increasing function
$f(\rho)$. Previous works considered various specific values of the exponent $m$. In \cite{BurgerDiFrancescoFranek} the case $m=2$ is studied and 
stationary solutions with compact support in one spatial dimension are characterized. The study in \cite{BurgerDiFrancescoFranek} makes extensive use of the interplay between energy minimizers of \eqref{def:entropyfunctional} and equilibria of \eqref{eq:aggeq}.
Topaz {\em et al.} \cite{TBL} investigate the case $m=3$ in one and higher dimensions, where the 
focus, again, is the characterization of the steady states. The authors find 
compactly supported steady states with steep edges which they call ``clumps".   

The results from \cite{BurgerDiFrancescoFranek} and \cite{TBL} are the
main motivation for this paper. We investigate the existence of
finite-size, compactly supported stationary states (``clumps") for
general power exponent $m>1$. Such biologically relevant equilibria have also been sought and studied in aggregation models with nonlocal repulsion \cite{BeTo2011,LeToBe2009}.
We show that such equilibria exist for
$m>2$ regardless of the size of the diffusion coefficient $\nu$.
Analytical results and numerical experiments suggest that these steady
states are the global minimizers of the energy, identified by
Bedrossian in \cite{Bedrossian}. However, the convergence in the
aggregation dominated regimes could be
arbitrarily slow, leading to metastable dynamics, as  already observed
in~\cite{TBL}. The case $1<m<2$ is more subtle, as dynamics depends
on the size of $\nu$, as well as on the spread of the initial data.
More precisely, for values of $\nu$ above a certain threshold (which
we can identify in a nonlinear eigenvalue problem), diffusion dominates and spreading to a trivial equilibria occurs. For values of $\nu$ below this value, clumps exist, but have only a limited basin of attraction. As shown in the numerical experiments, the spread of the initial data is important in determining whether the solutions approach 
these stationary states or disperse to 
infinity.

The paper is organized as follows: in Section \ref{sec:model} we
present briefly several derivations of the model \eqref{eq:main_evo}
and discuss its basic properties.
Section \ref{sect:ss-pf} is devoted to rigorous analysis of the
existence of compactly supported stationary solutions
to~\eqref{eq:aggeq}. These stationary states are computed 
with an iterative scheme in Section \ref{sec:numss}. In Section \ref{sect:dynamics},
 the long time  asymptotic behaviours of solutions to the evolution equation \eqref{eq:aggeq}
are investigated numerically.


\section{Model Derivation and Basic Properties} \label{sec:model}

\subsection{Model Derivation}
In the following we discuss two natural derivations of the general model \eqref{eq:main_evo}.

\paragraph{Microscopic derivation.} We consider a system of $N$ particles 
at positions $X_i(t)$, $i=1,\ldots,N$, in one dimension, with two kinds of interactions:
a {\em long-range} attraction, which we assume to be in a standard additive form, and a {\em nearest-neighbour} repulsion. For simplicity we assume that $X_1 < X_2 < \ldots < X_N$. Consider the model
\begin{equation}\label{eq:partsys}
\frac{dX_i}{dt} = - N R'\big(N(X_i-X_{i-1})\big) + N R'\big(N(X_{i+1}-X_i)\big) + \frac{1}N \sum_j G'(X_i-X_j),
\end{equation}
where $R$ denotes the {\em local} repulsion potential and $G$ the {\em global} attraction potential. The repulsive interaction only concerns nearest neighbours and is scaled via $N$ to obtain a meaningful locality. Note that we can also rewrite \eqref{eq:partsys} as a gradient flow for the energy
\[ E[X_1,\ldots,X_N] = \sum_i R(N(X_i-X_{i-1})) - \frac{1}{2N}
\sum_{i} \sum_{j \neq i} G(X_i-X_j). \]

In order to derive a continuum model, consider a discretization $u^N$ of the 
pseudo-inverse of the cumulative density distribution function $F$, with
$u^N(ih,t) = X_i(t)$ and $h = 1/N$. Then~\eqref{eq:partsys} is equivalent to
\begin{multline*}
\partial_t u^N(s,t)  = \frac{1}{h}\left[ -
    R'\left(\frac{u^N(s,t)-u^N(s-h,t)}{h}\right) +
    R'\left(\frac{u^N(s+h,t)-u^N(s,t)}{h}\right)\right] \cr
    + h\sum_j  G'\big(u^N(s,t)-u^N(jh,t)\big).
\end{multline*}
The limit $h \rightarrow 0$  naturally yields 
\begin{equation*}
\partial_t u(s,t)  = \partial_s \big( R'\big(\partial_s u(s,t)\big) \big) + \int_0^1 G'\big(u(s,t)-u(\sigma,t)\big)~d\sigma.
\end{equation*}
This equation for the pseudo-inverse $u=F^{-1}$ 
can be transformed back in a standard way (cf. \cite{Burger:DiFrancesco,MR2062430}) to an equation for the density distribution 
$\rho=\partial_x F = \frac{1}{\partial_s u}$:
\begin{equation}
\partial_t \rho = \partial_x \big( -
\partial_x R'(1/\rho)- \rho \partial_x G*\rho\big). 
\label{eq:micromacro}
\end{equation}

Now we immediately obtain a microscopic interpretation of \eqref{eq:main_evo}, by noticing that 
we can write \eqref{eq:micromacro}  in this form, provided 
$R''\big(1/\rho\big) =   \rho^3 f'(\rho)$, or equivalently,
$-\partial_x R'(1/\rho)=\rho \partial_x f(\rho)$.
In particular, the choice $f(\rho) = \nu \rho^{m-1}$ can be interpreted as the limit of a microscopic nearest-neighbour repulsive potential
$$ R(z) =  \frac{\nu}{m} |z|^{1-m}, $$
or a repulsive force proportional to $z  |z|^{-m-1}$. It is obvious that higher exponents in the nonlinear diffusion correspond to stronger local repulsion, and we also observe an interesting demarcation at $m=2$. For $m<2$ the microscopic repulsion is weak, i.e., it has an integral potential, while it becomes non-integrable for $m\geq 2$. We shall see different properties of the model for $m<2$ and $m>2$ on several instances in our analysis.

\paragraph{Fluid-dynamic derivation.} For an alternative derivation, which is quite standard for nonlinear diffusions, we consider directly the macroscopic compressible Euler equations with linear friction (friction coefficient scaled to one) and an additional nonlocal force, i.e.
\begin{align*}
\partial_t \rho + \nabla \cdot (\rho u ) &= 0 \\
 \partial_t u + u \cdot \nabla u &= - u + \nabla G*\rho - \frac{1}\rho \nabla p(\rho). 
\end{align*}
For a diffusive scaling of macroscopic type, i.e. $\tilde x = \epsilon x$, $\tilde t = \epsilon^2 t$, and $\tilde u = \epsilon^{-1} u$, the left-hand side in the second equation is of order $\epsilon^3$, while the right-hand side is of order $\epsilon$. Hence we 
obtain the leading order terms
\begin{equation*}
\partial_{\tilde t} \rho + \nabla_{\tilde x} \cdot (\rho \tilde u ) = 0,\qquad
  \tilde{u} = \nabla G*\rho - \frac{1}\rho \nabla p(\rho). 
\end{equation*}
and inserting $\tilde{u}$ into the first equation yields
\eqref{eq:main_evo} with the relation $
\rho f'(\rho) = p'(\rho)$.


\subsection{Gradient Flow Structure}
\label{sec:defG}
\textbf{Assumptions on the kernel $G$}. Given the focus of this paper, we list the properties of G only in one dimension. Throughout the paper the interaction kernel $G$ is assumed to satisfy
\begin{enumerate}
  \item $G\geq 0$ and $\mathrm{supp}(G)=\R$,
  \item $G\in W^{1,1}(\R)\cap   C^2(\R\setminus \{ 0\})$,
  \item $G(x)=g(|x|)$ for all $x\in \R$,
  \item $g'(r)<0$ for all $r>0$, $\lim_{r\rightarrow +\infty} g(r) =0$.
\end{enumerate}

Kernel $G$ having infinite support (Assumption 1) means that the attractive interactions are {\em global}. This is an important assumption used to conclude that the support of a steady state is connected. Assumption 2 concerns regularity properties on $G$ which are needed to pass derivatives inside the integral operator $G \ast \rho$, as well as in other estimates. In particular, Sobolev  embeddings imply $G \in L^\infty(\R)\cap C(\R)$.  The assumptions on $G$ could be potentially relaxed, but providing such sharp conditions  is not a purpose of the present paper. Note that pointy potentials such as $e^{-|x|}$ are included in the present theory. Assumption  3 is symmetry (isotropic interactions) and Assumption 4 means that kernel $G$ is purely attractive and interactions decay at infinity.

Using the recently developed techniques on gradient flows in metric spaces, in particular in the space of probability measures equipped with the Wasserstein metric (cf. \cite{ambrosiogiglisavare}), it is straightforward to analyze the well-posedness of the dynamic model. The evolution equation \eqref{eq:main_evo} can be written in the form
\begin{equation}
	\partial_t \rho = \partial_x \Big( \rho \partial_x \frac{\delta E[\rho]}{\delta \rho}\Big), \qquad 
\frac{\delta E[\rho]}{\delta \rho}=f(\rho)-G*\rho,
\end{equation}
which is the standard form for Wasserstein gradient flows~\cite{ambrosiogiglisavare} of the energy~\eqref{def:entropyfunctional}. 

If $ f(\rho)\rho - F(\rho)$  is nondecreasing
 and $G$ satisfies the above regularity assumptions, it is well-known that $E$ is geodesically $\lambda$-convex (cf. \cite{MR1812873}), and the existence and uniqueness of a dynamic solution $\rho \in L^\infty(0,T;{\cal P}(\R))$ follows. 
If $F(\rho)$ is given by \eqref{eqn:powerfF}, the decrease of the entropy yields an a-priori bound for the $L^m$-norm, hence $\rho \in L^\infty(0,T;L^m(\R))$. 


\subsection{Basic Properties of Stationary Solutions} 
\label{sect:statsol}

The central issue in this paper is to understand stationary states of \eqref{eq:aggeq} and their implications for the long time behaviours
of the dynamics. Using the more general form \eqref{eq:main_evo} in one dimension, equilibria $\rho$ satisfy a.e. in $\mathbb{R}$ the equation
\[
\rho \partial_x( f(\rho) - G \ast \rho) =0.
\] 

Due to the gradient flow formulation of the equation, such equilibria are closely linked to critical points of the energy. In particular, minimizers of the energy functional \eqref{def:entropyfunctional} on the manifold of probability measures are stationary solutions. 
We present here a few facts about the equilibria of \eqref{eq:main_evo} which can be derived relatively easily by generalizing results from \cite{BurgerDiFrancescoFranek}. These facts are not directly used in the paper, but they provide strong motivation for our subsequent studies. 
We do not present their proofs, except for the last one (see the Appendix). Below, $\mathcal{P}$ denotes the space of non-negative integrable functions of mass one. The nonlinear diffusion function $f(z)$ is assumed to be monotone, $C^1$ for $z>0$, and $\lim_{z \to 0}  f'(z)z < \infty$.

\begin{itemize} 
\item A stationary solution $\rho \in L^2\cap \mathcal{P}$ of  \eqref{eq:main_evo} is a stationary point for the energy functional $E$ defined in~\eqref{def:entropyfunctional}. 

\item A stationary solution $\rho \in L^2\cap \mathcal{P}$ of  \eqref{eq:main_evo} that has {\em connected} support, satisfies
\begin{equation}
\label{eqn:ss-hd-gen}
    f(\rho(x)) = \int_{\mathrm{supp}[\rho]} G(x-y)\rho(y) dy - C, \qquad \textrm{ for all } x\in \mathrm{supp}[\rho],
\end{equation}
with 
\begin{equation}
\label{eqn:C-gen}
C=-2E[\rho] + \int_{\mathrm{supp}[\rho]} (2 F(\rho)-f(\rho)\rho) dx.
\end{equation}

\smallskip

The two facts above may be shown in fact for general dimension. The remaining ones apply to {\em one} dimension only.
\item A stationary solution of  \eqref{eq:main_evo} in  one dimension has {\em connected} support.
\item A minimizer for the energy \eqref{def:entropyfunctional} under the constraint that the center of mass is zero, is necessarily symmetric and monotonically decreasing on $x>0$ by 
Riesz rearrangement inequality.
\item Consider a stationary solution $\rho$ of \eqref{eq:main_evo} with compact support. Then there exists a symmetric stationary solution $\rhot$ such that
\begin{equation*}
    E[\rhot] = E[\rho].
\end{equation*}
\end{itemize}

For power-law functions  given by \eqref{eqn:powerfF}, equations \eqref{eqn:ss-hd-gen}-\eqref{eqn:C-gen} display an interesting demarcation 
at $m=2$ where $C=-2E[\rho]$.  Indeed, we have 
\[ 2 F(\rho) - f(\rho)\rho = \nu\left(\frac{2}m -1\right) \rho^m, \]
which is positive for $m<2$ and negative for $m > 2$. 

Consider the possibility of a stationary solution with $\mathrm{supp}[\rho]$ having infinite measure.  For $m \geq 2 $ we can conclude from the finiteness of the energy that $\rho \in L^1(\R) \cap L^m(\R)$ and thus also $\rho \in L^{m-1}$. Hence, 
$$ C = {\nu} \rho^{m-1} -\int_{\mathrm{supp}[\rho]} G(x-y)\rho(y) dy \in L^1(\R), $$
which implies $C=0$.  Thus, $E[\rho] < 0$ for stationary solutions with unbounded support. 
In the case $m < 2$ it is not even clear that whether $C >0$, hence there could be stationary solutions with positive value of the entropy.

Motivated by the facts above, we focus in the rest of the paper on analytical and numerical investigations of  equilibria of \eqref{eq:aggeq} that have compact and connected support.


\section{Compactly Supported Stationary States}
\label{sect:ss-pf}

Based on previous works \cite{TBL,BurgerDiFrancescoFranek} and our own numerical investigations, there is a  class of steady states which is particularly relevant and important for the dynamics of \eqref{eq:aggeq}. We refer here to symmetric, monotone and compactly supported steady states.  Note also that solutions to \eqref{eq:aggeq} preserve mass, that is, $\int \rho(x,t) dx$ remains constant during the time evolution. Hence, equilibria can be considered of fixed mass, given by the initial configuration.

We are interested in finding a symmetric (even) density that has finite support $[-L,L]$, vanishes on the boundary of the support, and has unit mass. In mathematical terms we look for a solution of the equation (see \eqref{eqn:ss-hd-gen}):
\begin{equation}
\label{eqn:ss}
\nu \rho(x)^{m-1} = \int_{-L}^{L} G(x-y) \rho(y) dy -C, \qquad x\in [-L,L],
\end{equation}
where $\rho$ vanishes outside $[-L,L]$ and $\int \rho(x) dx =1$.
In particular, the continuity condition $\rho(L)=0$ implies that $C=G*\rho(L) = \int_{-L}^L G(L-y)\rho(y)dy$.

%


In formulating the integral equation for the steady state, we
included the requirement that solutions have unit mass. This is a
natural normalization that is inherited from the dynamics of the
model. However, in the existence of steady states investigated in
this section and for the alternative numerical calculations of the
Bessel potential $G(x)=e^{-|x|}/2$ in Section
\ref{sect:Morse}, it is more convenient to work with the normalization $\rho(0)=1$ instead. 
Hence, depending on the context, we may work with any of the two normalizations, and transfer results from one normalization to the other  using the following simple scaling argument. 

We present the scaling argument for one dimension, but it is straightforward to see that it holds identically in any dimension. 
Note that if  $\rho$ is a solution of the steady state equation \eqref{eqn:ss} that has the desired properties, then
\begin{equation}
\label{eq:normeq}
 \nu \rho(x)^{m-1} = G\ast\rho(x)-G\ast\rho(L),
\end{equation}
and consequently, $\rho_\lambda (x) = \lambda \rho(x)$ satisfies
\begin{equation}
\label{eq:renormeq}
 \nu\lambda^{2-m} \rho_\lambda(x)^{m-1}=
G\ast\rho_\lambda(x)-G\ast\rho_\lambda(L).
\end{equation}
We can choose $\lambda$ to pass from one normalization to another. Note that the support $[-L,L]$ is fixed through the transformation. If $\rho$ in~\eqref{eq:normeq} has unit mass, we take
$\lambda = \rho(0)^{-1}$, and $\rho_\lambda$ in ~\eqref{eq:renormeq} 
has the normalization $\rho_\lambda(0)=1$. On the other hand,
if $\rho$ in~\eqref{eq:normeq} has the normalization
$\rho(0)=1$, then by choosing $\lambda = 
(\int \rho(x)dx)^{-1}$, $\rho_\lambda$ in~\eqref{eq:renormeq}
has unit mass. Note however that, in making this transformation, the diffusion coefficient $\nu$  changes as well.



\subsection{Existence of Stationary Solutions in One Dimension}
\label{subsect:fp}

Using the symmetry of the solution and the fact the $\rho$ vanishes at $x=L$ , equation \eqref{eqn:ss} can be written as
\begin{equation}
\label{eqn:ssG}
\nu \rho(x)^{m-1}(x) = \mathcal{G}_L[\rho](x), \qquad x \in [0,L],
\end{equation}
where
\begin{equation}
\label{eqn:G}
\mathcal{G}_L[\rho](x) = \int_0^L [G(x-y)+G(x+y)-G(L-y)-G(L+y)] \rho(y) dy.
\end{equation}
Case $m=2$ (when the problem becomes {\em linear}) was investigated in \cite{BurgerDiFrancescoFranek}. The approach there was to take a derivative in \eqref{eqn:ssG} and transform the equation into an eigenvalue problem for $\rho'(x)$. The main advantage is that the kernel of the resulting integral operator is positive and enables the use of the Krein-Rutman theorem,
while the kernel of $\mathcal{G}_L$ defined in~\eqref{eqn:G} is not. We follow a similar approach here, in dealing with a general exponent $m$. However, the problem is {\em nonlinear} now, and the methodology from \cite{BurgerDiFrancescoFranek} has to be adapted and significantly extended.

Take a derivative in  \eqref{eqn:ssG} to get: 
\begin{equation}\
\label{eq:ss3}
 \nu (m-1) \rho(x)^{m-2}\rho'(x) =  \mathcal{H}_L[\rho'](x), \qquad x \in [0,L],
\end{equation}
where 
\begin{equation}
\label{eqn:H}
 \mathcal{H}_L[\rho'](x)=\int_0^L [G(x-y)-G(x+y)] \rho'(y) dy.
\end{equation}
Although $\rho'(x)$ can be infinite at $x=L$ (true for $m>2$), $\rho'$ 
is integrable and the equation~\eqref{eq:ss3} is still well-defined. 

In the following we shall prove the existence of compactly supported stationary states with arbitrary support size $L$ by a combination of the Krein-Rutman theorem with fixed point techniques. The first step is to ``freeze" $\rho$ and study the eigenvalue problem \eqref{eq:ss3} in $u=\rho'$. Krein-Rutman theorem (strong version) can be used to argue that a unique nonpositive eigenfunction $u$ exists, and hence we can define a map that takes $\rho$ into $u$. The next step is to consider a second map, which maps $u$ into its primitive, and show that the composition of the two maps has a fixed point. Such a fixed point is a solution of \eqref{eq:ss3}.

We point out the transition at $m=2$ for the study of solutions of \eqref{eq:ss3} performed here. Since $\rho$ vanishes at the boundary, $\rho(x)^{m-2}$ is either zero ($m>2$) or infinity ($m<2$) at $x=L$. Consequently, since the right-hand-side of \eqref{eq:ss3} is finite at $x=L$, $\rho'(L)$ is either infinite ($m>2$) or zero ($m<2$). We will treat the two cases separately.

 
We first state the strong version of the Krein-Rutman theorem used in the arguments.
 
\begin{thm}
[Krein--Rutman Theorem, strong version]\label{thm:KR}
Let $X$ be a Banach space, let $K\subset X$ be a \emph{solid cone}, i.e. such that $\lambda K\subset K$ for all $\lambda \geq 0$ and $K$ has a nonempty interior $K_0$. Let $T$ be a compact linear operator, which is strongly positive with respect to $K$, i.e. $T(K) \subset K_0$.
Then the spectral radius $r(T)$ is strictly positive and $r(T)$ is a simple eigenvalue with an
eigenvector $v \in K_0$. There is no other eigenvalue with a corresponding eigenvector $v \in K$.
\end{thm}

%
%

 
Another useful tool in the following is a comparison result on the leading eigenvalue, which can be derived easily from the Krein-Rutman theorem:
\begin{lem} \label{comparisonlemma} Let $\mathcal{K}$ be a compact linear integral operator 
 from $C([0,L)) \to C([0,L)) $ with positive continuous kernel and $v$ a positive continuous function on $[0,L)$        such that $\mathcal{K} v \geq \mu v$. Then the leading 
        eigenvalue $\lambda(\mathcal{K})$ (from the Krein-Rutman theorem) is greater or equal $\mu$.
\end{lem}
\proof
Assume that the maximal eigenvalue $\lambda$ is smaller than $\mu$. Since $K$ is an integral operator with positive kernel, its $L^2$-adjoint operator is a positive operator as well and hence, the Krein-Rutman theorem also implies the existence of a nonnegative function $u^* \in C([0,L])$ such that 
$$ K^* u^* = \lambda u^*. $$
On the other hand, due to the positivity of $v$ we have
$$ 0 < (\mu-\lambda) \int_0^L v(x) u^*(x)~dx \leq  \int_0^L [(Kv)(x) u^*(x) - v(x)(K^*u^*)(x)]~dx = 0, $$
which is a contradiction.
\endproof

\subsubsection*{Case $m > 2$}

In the case $m > 2$ we rewrite the problem \eqref{eq:ss3} for the derivative in the form 
\begin{equation}
\label{eq:ss4}
 \nu (m-1) \rho'(x) = \rho(x)^{2-m} \mathcal{H}_L[\rho'](x), \qquad x \in [0,L].
\end{equation}
In order to avoid potential issues with dividing by zero (due to $\rho(L)=0$) we start with a regularized problem of the form
\begin{equation}\
\label{eq:ss5}
 \nu (m-1) \rho'(x) = (\rho(x)+\delta)^{2-m} \mathcal{H}_L[\rho'](x), \qquad x \in [0,L].
\end{equation}

Our final goal is to prove the existence of a solution $\rho$ in the following subset of $C([0,L])$, the set of continuous functions on $[0,L]$:
\begin{equation}
\label{setD}
	{\cal D} = \big\{ \rho \in C([0,L])\mid \rho \text{ nonnegative and nonincreasing}, \rho(0)=1, \rho(L)=0 \big\}.
\end{equation}
Note that ${\cal D}$ is a bounded set, since any $\rho \in {\cal D}$ attains its maximum at $x=0$ and hence $\Vert \rho \Vert_{C([0,1])} = 1$.

We first define the map that takes $\rho$ to the leading eigenfunction of the eigenvalue problem
\begin{equation} \label{eq:eigenvaleu1}
	\lambda u = (\rho+\delta)^{2-m} \mathcal{H}_L[u].
\end{equation}
The existence and uniqueness of the leading eigenvalue follows ideas from \cite{BurgerDiFrancescoFranek}, where the strong version of the Krein-Rutman theorem was used to study the eigenvalue problem. 

A major difficulty is that we cannot simply work in the space $C([0,L])$ and the cone of nonpositive continuous functions, i.e. 
$	\tilde K = \{u \in C([0,L]) ~|~u \leq 0 \}$,
 since $v=(\rho+\delta)^{2-m} \mathcal{H}_L[u]$ always vanishes at $x=0$ and hence it is not in the interior of $\tilde K$. To avoid this issue one can work instead with the cone of nonpositive functions in the subspace of  $C^1([0,L])$ defined by functions with vanishing left boundary value. The control of the derivative and the favourable properties of the operator (which yields a definite sign of the derivative at $0$) help to apply the strong version of the Krein-Rutman theorem. More precisely we 
use (see also \cite{BurgerDiFrancescoFranek} for further discussion) the space
$$X:=\{u \in C^1[0,L]\mid u(0)=0\},$$ 
with the cone 
$$K:=\{u \in X\mid u \leq 0 \}.$$ 
As in \cite{BurgerDiFrancescoFranek} one can see that $v$ is now in the set
$$ H:= \{ v \in X ~|~ v'(0) < 0, v(x) < 0 \text{ for } x > 0 \}, $$
which is a subset of the interior of the cone $K$. Then we can apply the Krein-Rutman theorem to obtain:
 
\begin{lemma} \label{kreinrutmanlemma1}
For given $\rho \in C^1([0,L])\cap {\cal D}$ and $\delta > 0$ there exists a unique maximal eigenvalue $\lambda > 0$ and a unique nonpositive $u \in C^1([0,L])$ with $u(0)=0$ and  $\int_0^L u(x)~dx = -1$ satisfying \eqref{eq:eigenvaleu1}. Moreover, 
\begin{equation}
\label{eqn:lwb}
	\lambda \geq (1+\delta)^{2-m}  \mu,
\end{equation}
where $\mu$ is the maximal eigenvalue of the positive linear operator ${\cal H}_L$.
\end{lemma}
\begin{proof}
The existence and uniqueness of a maximal eigenvalue with eigenfunction $u$ as above follows from Theorem \ref{thm:KR} using the cone $K$.  The eigenfunction $u$ has been normalized to have integral $-1$.

In order to obtain a lower bound on $\lambda$ we employ Lemma \ref{comparisonlemma} with $\mu$ being the principal eigenvalue of $\mathcal{H}_L$, which exists and is positive due to the Krein-Rutman theorem, and $v$ an associated nonnegative eigenfunction. Then
$$ \mathcal{K}v = (\rho + \delta)^{2-m} \mathcal{H}_L[v]= (\rho+ \delta)^{2-m} \mu v
\geq (1+\delta)^{2-m} \mu v. $$
Hence, we conclude
$$ \lambda \geq (1+\delta)^{2-m} \mu.$$
\end{proof}

Note that with the lower bound on $\lambda$ we can estimate the supremum norm of the nonpositive eigenfunction $u$ using its normalized $L^1$-norm. Indeed,
$$  |u(x)| \leq \frac{\lambda}{(1+\delta)^{2-m} \mu}  |u(x)| =  \frac{(\rho + \delta)^{2-m}}{(1+\delta)^{2-m} \mu}  {\cal H}_L[-u] \leq \frac{\delta^{2-m}}{(1+\delta)^{2-m} \mu}   2\max_{z \in \mathbb{R}} |G(z)| =: C_\delta. $$

Due to the setup of the Krein-Rutman theorem we have to assume $\rho \in C^1([0,L])$, which in turn affects the setup of the fixed-point argument, as it now needs to be carried out in $C^1([0,L])$ rather than $C([0,L])$. Note that $C^1([0,L])$ is counter-intuitive at a first glance since we expect stationary solutions of the original problem to have infinite derivative at $x=L$ in the case $m>2$. However, the derivative will be finite for positive $\delta$ used in the fixed point argument, which is also encoded somehow in the fact that $C_\delta \rightarrow \infty$ as $\delta \rightarrow 0$. In the limiting procedure $\delta \rightarrow 0$ we will subsequently use a limit of solutions in ${\cal D}$ and convergence in $C([0,L])$.

We construct a fixed-point operator in two steps on the convex set
\begin{equation}
	{\cal D}_\delta = \big\{ \rho \in C^1([0,L]) \mid \Vert \rho' \Vert_{C([0,1])} \leq C_\delta, \rho \text{ nonnegative and nonincreasing}, \rho(0)=1, \rho(L)=0 \big\}.
	\label{setD1}
\end{equation}
Note that ${\cal D}_\delta \subset {\cal D}$. For given $\rho \in \mathcal{D}_\delta$ define
\begin{equation*}
 \label{def:F1}
	\begin{array}{lrcl}{\cal F}_1: & \mathcal{D}_\delta & \rightarrow & C^1([0,L]) \\ 
	             & \rho & \mapsto & u, \end{array}
\end{equation*}
where $u$ is the unique \emph{nonpositive}  eigenfunction for the leading eigenvalue of \eqref{eq:eigenvaleu1}, which satisfies
 $\int_0^L u(x)~dx = -1$.
Then consider the map
\begin{equation} 
\label{def:F2}
	\begin{array}{lrcl}{\cal F}_2: &  C^1([0,L]) & \rightarrow & C^1([0,L]) \\ 
	             & u & \mapsto & \rho, \end{array}
\end{equation}
with 
\begin{equation}
\label{eqn:rhoint}
	\rho(x) = - \int_x^L u(y)~dy.  
\end{equation}
Therefore, by the normalization of $u$ and the above estimate on its supremum norm, 
solutions of \eqref{eq:ss5} are fixed points of the map ${\cal F}_2 \circ {\cal F}_1$. 



We now investigate the properties of the two maps ${\cal F}_1$ and $ {\cal F}_2$.
\begin{lemma} \label{F1lemma}
The operator ${\cal F}_1: {\cal D}_\delta \rightarrow C^1([0,L])$ is well-defined and continuous. Moreover, 
$u={\cal F}_1(\rho)$ is nonpositive with $u(0)=0$ and satisfies $\int_0^L u(x)~dx = -1$.
\end{lemma}\proof
From Lemma \ref{kreinrutmanlemma1} we obtain the well-definedness and the properties of $u$. It remains to verify the continuity of $\mathcal{F}_1$. To do so we symmetrize the eigenvalue problem by the transform
$ v = u (\rho + \delta)^{m/2-1}$, which is obviously continuous on $C^1([0,L])$. Hence, it suffices to prove the continuity of the map $\rho \mapsto v$. Then $\lambda$ is the unique leading eigenvalue of the self-adjoint operator
\begin{equation}
	{\cal K}_{\rho,L}[v] = (\rho+\delta)^{1-m/2} {\cal H}_L[(\rho+\delta)^{1-m/2}v],
\end{equation}
which depends on the integral kernel, and consequently on $\rho$, in a locally Lipschitz continuous way. 

For given $\rho_1$ and $\rho_2$, and $(\lambda_1,v_1)$, $(\lambda_2,v_2)$ eigenpairs of  ${\cal K}_{\rho_1,L}$, ${\cal K}_{\rho_2,L}$, respectively, we further have
\begin{equation*}
	\lambda_1 (v_1 - v_2) - {\cal K}_{\rho_1,L}[v_1-v_2] = (\lambda_2 - \lambda_1) v_2 + ({\cal K}_{\rho_1,L}[v_2]- {\cal K}_{\rho_2,L}[v_2]).
\end{equation*}
Since the right-hand side is orthogonal to each element in the null-space of the operator $\lambda_1 - {\cal K}_{\rho_1,L}$ (which consists only of multiples of $v_1$), the Fredholm alternative implies the existence of a unique solution $v_1-v_2$, which depends continuously on the right-hand side (also in the supremum norm, cf. \cite{Engl}). Hence we have for some constant $C$ depending on $\rho_1$,
$$ \Vert v_1 - v_2 \Vert_\infty \leq C |\lambda_1 - \lambda_2| \Vert v_2 \Vert_\infty 
+ C \Vert {\cal K}_{\rho_1,L}[v_2]- {\cal K}_{\rho_2,L}[v_2] \Vert_\infty.  $$
For the first term on the right-hand-side we already know the local Lipschitz continuity on $\rho$, while the local Lipschitz continuity of the second term is a straightforward computation.
Finally, we can differentiate the eigenvalue equation with respect to $x$ and obtain an estimate in the norm of $C^1([0,L])$ in an analogous way.
\endproof

\begin{lemma} \label{F2lemma}
The linear operator ${\cal F}_2$ is well-defined by \eqref{def:F2} and compact. Moreover, each function $u$ with $\int_0^L u(x)~dx = -1$ and $u(x) < 0$ for $x \in (0,L]$ maps to a function $\rho \in {\cal D}_\delta$.
\end{lemma}
\proof
The proof is immediate, using \eqref{eqn:rhoint} and \eqref{setD1}.
\endproof

Putting these properties together we can show:
\begin{thm}
For each $\delta > 0$ sufficiently small and $L >0$, there exists a solution $\nu_\delta > 0$ and $\rho_\delta \in {\cal D}_\delta \subset {\cal D}$ of \eqref{eq:ss5}.
\end{thm}
\proof
A density $\rho_\delta \in {\cal D}_\delta$ satisfies \eqref{eq:ss5} if it is a fixed-point of the map ${\cal F}_2 \circ  {\cal F}_1$, i.e.,
\begin{equation*}
	\rho_\delta = {\cal F}_2 ( {\cal F}_1(\rho_\delta)).  
\end{equation*}
From Lemma \ref{F1lemma} and \ref{F2lemma} we conclude that ${\cal F}_2 \circ {\cal F}_1$ is continuous, compact, and maps the convex and bounded set ${\cal D}_\delta$ into itself. Thus, the assumptions of the Schauder fixed-point theorem are satisfied and we conclude the existence of a solution $\rho_\delta \in {\cal D}_\delta$. The existence of $\nu_\delta >0 $ then comes from the fact that
${\cal F}_1(\rho_\delta)$ is a Krein-Rutman eigenfunction of \eqref{eq:eigenvaleu1} and $\nu_\delta$ is obtained from the spectral radius: $\nu_\delta (m-1) = \lambda$. Thus, we obtain for each $\delta$ a solution $\rho_\delta \in {\cal D}_\delta$. 
\endproof

The final task remaining is to send $\delta$ to zero. 
 We proceed in several steps.  As mentioned above, the derivative $\rho_\delta'$ approaches infinity at the boundary $x=L$ as $\delta \to 0$ and hence is not suitable for performing the limit.
For this reason we first integrate \eqref{eq:ss5} to write the eigenvalue problem in terms of $\rho_\delta$ 
only. We write \eqref{eq:ss5} as
\begin{align}\label{eq:mdev1}
    \nu_\delta \frac{d}{dx}(\rho_\delta(x)+\delta)^{m-1}   &=
    \frac{d}{dx}\int_0^L \big(G(x+y)+G(x-y)\big)\rho_\delta(y)dy,    
\end{align}
and integrate from $L$ to $x$ to get
\begin{equation}\label{eq:mint}
    \nu_\delta ( \rho_\delta(x) + \delta)^{m-1}
    = \mathcal{G}_L[\rho_\delta](x)  + \nu_\delta \delta^{m-1},
\end{equation}
where $\mathcal{G}_L$ was defined in \eqref{eqn:G}.

Evaluating at $x=0$, we can get an uniform upper bound on $\nu_\delta$:
\begin{align}\label{eq:eb}
    \nu_\delta \leq \nu_\delta ((1+\delta)^{m-1}  - \delta^{m-1}) \leq \mathcal{G}_L[\rho_\delta](0)
    \leq 
    \int_0^L \big| G(y)+G(-y)-G(L+y)-G(L-y)\big| dy  .
\end{align}
Thus, the sequence $\nu_\delta$ has a convergent subsequence. In order to obtain 
convergence of a subsequence of $\rho_\delta$ we need some semicontinuity, which on the other
hand relies on the lower bound \eqref{eqn:lwb} from Lemma \ref{kreinrutmanlemma1}.

We are now ready to prove the main existence result for the case $m>2$.
\begin{thm}[Existence of compactly supported equilibria for $m > 2$] 
For each $L > 0$, there exists $\nu > 0$ and $\rho \in {\cal D}$ solving \eqref{eqn:ssG}.
\end{thm} 
\proof
We start with a subsequence $(\nu_k,\rho_k)$ such that $\nu_k$ converges to some $\nu \geq 0$, which exists due to the above arguments.  
We employ Lemma \ref{kreinrutmanlemma1} to obtain
$$ (m-1) \nu_\delta = \lambda(\mathcal{K})\geq (1+\delta)^{2-m} \mu, $$
which yields a uniform lower bound on $\nu_\delta$ as $\delta$ tends to zero. Subtracting the
equation for $\rho_\delta$ (see \eqref{eq:mint}) evaluated at $x$, respectively $y$, yields
\begin{align*}
    \nu_\delta (\rho_\delta(x)+\delta)^{m-1} - \nu_\delta(\rho_\delta(y)+\delta)^{m-1}
    = \int_0^L\big[ G(x+z)+G(x-z)-G(y+z)-G(y-z)\big] \rho_\delta(z)dz.
\end{align*}
Using the smoothness of $G$ and the lower bound for $\nu_\delta$, we estimate
\begin{equation}\label{eq:ec}
    \big|\rho_\delta(x)-\rho_\delta(y)\big|^{m-1}\leq    
    \big|(\rho_\delta(x)+\delta)^{m-1}-(\rho_\delta(y)+\delta)^{m-1}  \big|
    \leq C|x-y|
\end{equation}
for some positive constant $C$ independent of $\delta$. 

Hence, $\rho_\delta$ is uniformly H\"older continuous, in particular equicontinuous. With the Arzela-Ascoli theorem we deduce the existence of a convergent subsequence $(\nu_{\delta_k},\rho_{\delta_k})$ to 
$(\nu,\rho)$ satisfying \eqref{eqn:ssG} (the limiting equation of~\eqref{eq:mint}).

\endproof


\subsubsection*{Case $1 < m < 2$}

In the case $m < 2$ we use the transformation of variables from $\rho$ to $v:=\rho^{m-1}$ before differentiating, i.e., we analyze the problem
\begin{equation} \label{vequation}
\nu v(x) = \mathcal{G}_L[v^{1/(m-1)}](x), \qquad x \in [0,L].
\end{equation}
For the existence proof we use a very similar strategy based on a fixed-point formulation and 
the equation for the derivative of $v$ interpreted as an eigenvalue problem; hence we only give a sketch of the proof. 

The derivative $w=v'$ satisfies
\begin{equation}\label{wequation}
\nu (m-1) w(x) = \mathcal{H}_L[v^\beta w](x), \qquad x \in [0,L].
\end{equation}
with $\beta = \frac{2-m}{m-1}$ and $\mathcal{H}_L$ defined as in \eqref{eqn:H}.

Define ${\cal F}_3: {\cal D} \rightarrow C^1([0,L])$ as the map from $v$ to the eigenfunction $w$ of the leading eigenvalue of \eqref{wequation}, which is again normalized as a nonpositive function with $w(0)=0$ and
$\int_0^L w(x)~dx = -1$. Note that $v$ appears in \eqref{wequation} under the integral sign, hence enough regularity is obtained with $v$ being just continuous. Solutions of \eqref{wequation} are fixed points of the map ${\cal F}_4 \circ {\cal F}_3$, i.e.,
\begin{equation}
	v= {\cal F}_4 ({\cal F}_3(v)),
\end{equation}
with 
\begin{equation} 
\label{def:F4}
	\begin{array}{lrcl}{\cal F}_4: &  C^1([0,L]) & \rightarrow & C([0,L]) \\ 
	             & u & \mapsto & \rho, \end{array}
\end{equation}
given by \eqref{eqn:rhoint}.

With an analogous proof as for ${\cal F}_1$ we obtain
\begin{lemma} \label{F3lemma}
The operator ${\cal F}_3: {\cal D} \rightarrow C^1([0,L])$ is well-defined and continuous. Moreover, 
$w={\cal F}_3(v)$ is nonpositive with $w(0)=0$ and satisfies $\int_0^L w(x) dx = -1$.
\end{lemma}
The map ${\cal F}_4$ is the concatenation of ${\cal F}_2$ and an embedding operator. Using the above properties of ${\cal F}_2$ (see Lemma \ref{F2lemma}) we conclude that ${\cal F}_4 \circ {\cal F}_3$ is continuous and compact and maps ${\cal D}$ into itself. Thus we conclude from Schauder's fixed point theorem the existence of $\nu > 0$ and $w \in C^1([0,1])$ being the derivative of $v \in {\cal D}$ such that \eqref{wequation} is satisfied. Since $v$ is continuously differentiable and the exponent $\beta$ is positive, it is straightforward to justify integration of this equation and conclude the existence of $\nu > 0$ and $v$ solving \eqref{vequation}. 
Thus we obtain the following result: 

\begin{thm}[Existence of compactly supported equilibria for $1 < m < 2$] 
For each $L > 0$, there exists $\nu > 0$ and $\rho \in {\cal D}$ solving \eqref{eqn:ssG}.
\end{thm}
 
\begin{rem}
We finally explain why we choose to work in this case with a transformed variable and not with the original $\rho$ formulation. In principle, a fixed point approach could be constructed as in the case $m>2$, but it is more difficult to justify the way back from the equation for the derivative $u$ to the density $\rho$. For $m<2$ we end up with $\rho^{m-2}$ as a factor of $\partial_x \rho$, and since $m<2$ we cannot justify its well-definedness in cases where $\rho = 0$. This however seems to be rather a technical issue. More severely is the fact that we find $\rho'(L)=0$ in this case. Thus, there is no argument that $\rho$ is positive close to $L$ and the fixed-point proof in the $\rho$ variable would only yield the existence of a solution whose support is a subset of $[0,L]$. In the formulation with the new variable $v$ we know that $v$ is decreasing and $v'(L)$ is positive, so the support is exactly $[0,L]$.  
\end{rem}


\subsection{Further Properties of Compactly Supported Stationary Solutions}
 
We finally discuss some finer properties of compactly supported stationary solutions. A first observation concerns the behaviour at $x=L$.  From \eqref{eqn:ssG}-\eqref{eqn:G}, using the Lipschitz-continuity of $G$ we can find a constant depending on $G$ and $\nu$ only, such that
\begin{equation*}
	\rho(x) \leq C (L-x)^{1/(m-1)}.
\end{equation*}
Equation above yields another change at $m=2$.  The solution is H\"older continuous for $m > 2$ (as the Barenblatt solutions of the porous medium equation), while it is even differentiable, with vanishing derivative at $x=L$, for $m < 2$. With $m \rightarrow 1$, more and more derivatives at $L$ become zero. 

Another interesting question is the potential concavity of the stationary state. For this sake we compute a formula for the second derivative by differentiating \eqref{eq:ss3} on the support of $\rho$:
\begin{equation} 
\label{secondderivative}
 \nu (m-1) (m-2) \rho(x)^{m-3}(\rho'(x))^2 + \nu (m-1) \rho(x)^{m-2}\rho''(x) =  \int_0^L [G'(x-y)-G'(x+y)] \rho'(y) dy
\end{equation}
Evaluate at the origin $x=0$ and use $\rho'(0)=0$, $\rho(0)=1$ to get
\begin{equation}
\nu (m-1)\rho''(0) =  \int_0^L [G'(-y)-G'(y)] \rho'(y) dy.
\end{equation} 
Since $G'$ is positive for negative argument and vice versa we obtain, also using the negativity of $\rho'$, that $\rho''(0)<0$. Thus, $\rho$ is locally concave at $x=0$, which is not surprising since we expect to have its maximum there. 

The range of concavity of a stationary solution depends on $m$ as well as on the concavity of the kernel $G$. As already argued in \cite{BurgerDiFrancescoFranek}, the right-hand-side in \eqref{secondderivative} is nonpositive if $G$ is concave on $(-2L,2L)$, i.e., if the support is smaller than half the range of concavity of $G$. If $m>2$ this immediately implies that $\rho''$ is negative on the support of $\rho$, because $\nu (m-1) (m-2) \rho^{m-3}(\rho')^2$ is positive. In the case $m < 2$ we already know that $\rho'(L)=0$, which does not allow $\rho$ to be concave on the whole support.
 


\section{Numerical Calculation of the Steady States} 
\label{sec:numss}
According to the results of the previous section, steady states exist for any fixed domain size $L$. Now we turn our attention to the qualitative features of these equilibria, in particular to the dependence of the diffusion coefficient $\nu$  on $L$ and the limit $L\to \infty$.
To this purpose, we develop an iterative method to solve the eigenvalue problem  \eqref{eq:ss3} with general interaction potentials. The ultimate goal is to characterize  the long time behaviours of solutions to the evolution equation~\eqref{eq:aggeq}, and this approach is much more effective than solving \eqref{eq:aggeq} directly (see Section \ref{sect:dynamics}).

\subsection{Iterative Scheme} 
\label{subsect:intnum}


When $m \neq 2$, the nonlinear eigenvalue problem 
\eqref{eq:ss3} can still be solved iteratively,
though not in one single step as for $m=2$ \cite{BurgerDiFrancescoFranek}. The analytical study of the previous section suggests the following iterative scheme for the steady state. 
Fix $L$ and  a nonnegative function $\rho_0$ (the initial guess) on $[0,L]$. For $k=0,1,2,...$, perform the following two steps:
\begin{enumerate}
 \item Compute the leading eigenvalue  $\lambda_{k}$ and the the corresponding eigenfunction $e_k$  of the eigenvalue problem (see \eqref{eq:ss3}):
\begin{equation}
\label{eqn:step1}
 \lambda_{k} (m-1) \rho_{k}(x)^{m-2} e_{k}(x) = \mathcal{H}_L[e_{k}](x), \qquad x \in[0,L].
\end{equation}
The eigenfunction $e_k$, computed up to a multiplication constant, is assumed to be nonpositive (this convention is made for consistency with considerations in Section \ref{subsect:fp}).
\item Find $\rho_{k+1}$ by integration and normalization to unit mass:
\[
 \rho_{k+1}(x) = \frac{\int_{x}^L e_k(x)dx}{2\int_0^L\int_x^L e_k(s)dsdx}
=\frac{\int_{x}^L e_k(x)dx}{2\int_0^L xe_k(x)dx}.
\]
\end{enumerate}
In the limit when $k \to \infty$, $\lambda_k$ 
converges to $\nu$ (depending on $L$) and $\rho_k$ 
converges to a solution of~\eqref{eqn:ss}.

To implement the algorithm numerically, we set an uniform grid $0=x_0<x_1<\cdots <x_N=L$ on $[0,L]$. We assume that $e_k$ is piecewise constant on the grid, that is,  $e_k(x) = g_i$ on $[x_i,x_{i+1})$, $i=0, 1, \cdots, N-1$. Given the piecewise constant function $e_k(x)$ on $[0,L]$, $\rho_{k+1}$ in Step 2 is piecewise linear.  To avoid however a possible singularity at $x_N=L$ (see Section  \ref{subsect:fp} for  $m>2$), in Step 1 we evaluate the density at middle points $x_{i+\frac{1}{2}}=(x_i+x_{i+1})/2$, that is,  $\rho_k(x_{i+\frac{1}{2}}) = \big(\rho_k(x_i)+\rho_k(x_{i+1})\big)/2$. 
Hence equation \eqref{eqn:step1} evaluated at $x_{i+\frac{1}{2}}$ then reads (with subindex $k$ omitted for convenience)
\begin{equation}
\label{eqn:evd}
\lambda (m-1) \rho(x_{i+\frac{1}{2}})^{m-2} g_i= \sum_{j=0}^{N-1} M_{ij} g_j, \qquad j=0,1, \cdots, N-1,
\end{equation}
where $M_{ij}>0$ are approximations (via trapezoidal rule) of 
$
 \int_{x_j}^{x_{j+1}} \left[ G(x_{i+1/2}-y) - G(x_{i+1/2}+y)\right] dy.
$
With $\rho(x_{i+\frac{1}{2}})$ known from the previous iteration, 
the leading eigenvalue problem \eqref{eqn:evd} to find  $\lambda$ and $\mathbf{g}=(g_0,\cdots,g_{N-1})$ can be solved with standard algorithms. We terminate the algorithm when the change in $\rho_k$ (calculated  in the supremum norm) is smaller than $10^{-8}$.

Uniqueness of equilibria is not guaranteed by the analytical considerations from Section \ref{subsect:fp}. However, all numerical investigations we performed showed convergence of the algorithm described above to a unique solution of \eqref{eqn:ss}. Finally, we point out that the numerical algorithm also works for compactly supported kernels such as $G(x)=\max(1-|x|,0)$, which violate assumption 4 (strict  monotonicity)  in Section~\ref{sec:defG}. For such kernels though, the steady states of the evolution equation \eqref{eq:aggeq} are {\em not} unique, as multiple disconnected and non-interacting bumps can exist.


\subsection{Qualitative Properties of the Steady States} 
\label{subsect:qprop}

We start with the dependence of the diffusion coefficient $\nu$ on the  domain size $L$.
Figure~\ref{fig:bifur} corresponds to the Gaussian kernel $G(x)=e^{-|x|^2/2}/\sqrt{2\pi}$, while qualitatively similar plots were obtained for other kernels such as the Bessel potential $G(x)=e^{-|x|}/2$ or the hat function $G(x)=\max(1-|x|,0)$. In the limit $L \to \infty$, $\nu$ goes to infinity when $m>2$, while for $m \in (1,2]$, $\nu$ approaches a constant value. Both behaviours will be  commented in detail below. 

\begin{figure}[htp]
 \begin{center}
  \includegraphics[totalheight=0.3\textheight]{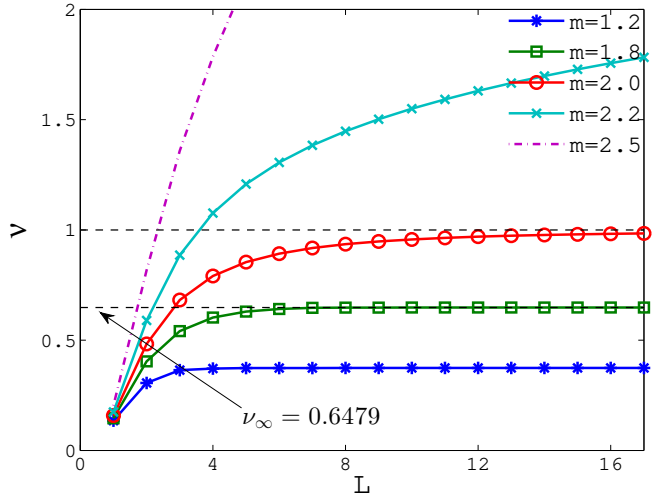}
 \end{center}
\caption{The dependence of the diffusion coefficient $\nu$
on the size $L$ of the support of the steady state for different exponents $m$.  When $m>2$, $\nu$ increases indefinitely with $L$, while for $1<m\leq 2$, $\nu$ approaches a finite limit $\nu_\infty$ as $L\to\infty$ ($\nu_\infty=\|G\|_{L^1}$ for $m=2$~\cite{BurgerDiFrancescoFranek} and specified by~\eqref{eq:nleig} for $m<2$). The plot corresponds to the Gaussian kernel $G(x)=e^{-|x|^2/2}/\sqrt{2\pi}$, but the result seems to be generic for the class of potentials considered in this paper.}
\label{fig:bifur}
\end{figure}

The steady states shown in Figure~\ref{fig:sstate}
exhibit qualitatively different behaviours too,  depending on whether $1<m<2$ or $m>2$. We discuss the two cases separately.

\subsubsection*{Case $m > 2$}
When $m>2$  (see Figure~\ref{fig:sstate}(a)), the steady states are decreasing and concave down, and spread and decay to zero when 
$L$ increases to infinity. Moreover the numerical solution has a very peculiar form in this limit: it is close to a
constant on the support, with a sharp drop near the boundary at $L$. Hence, for $L$ large, one can approximate 
the density profile  by a rescaled characteristic function $\rho(x) \approx 
\frac{1}{2L}\chi_{[-L,L]}$ (to preserve unitary mass). Then $C=G*\rho(L) \approx 0$, 
\[ 
 \frac{1}{2L}\int_{-L}^L G(x-y) dy
\approx \frac{1}{2L}\|G\|_{L^1},
\]
and the steady state equation~\eqref{eqn:ss} yields
$ (2L)^{1-m} \nu\approx (2L)^{-1}\|G\|_{L^1}$. This implies the scaling 
law 
\begin{equation}
\label{eqn:scaleL}
L\sim \nu^{1/(m-2)}, \qquad \text{ for large } L.
\end{equation} 
A similar result, but regarding  the dependence of $L$ on large mass (with $\nu$ fixed) was derived in~\cite{TBL}.

To summarize, when $m>2$, for {\em any} diffusion coefficient $\nu>0$, there exists a compactly supported steady state that solves \eqref{eqn:ss}. Numerical results indicate that this steady state is unique and simulations of the evolution equation \eqref{eq:aggeq} (see Section \ref{sect:dynamics}) indicate that these equilibria are {\em global} attractors  for the dynamics. The larger the diffusion coefficient $\nu$, the larger the size $L$ of the support, with scaling given by \eqref{eqn:scaleL}. Equilibria look like rescaled characteristic functions as $\nu$ (and $L$) become large.

\begin{figure}[htp]
 \begin{center}
  \includegraphics[totalheight=0.25\textheight]{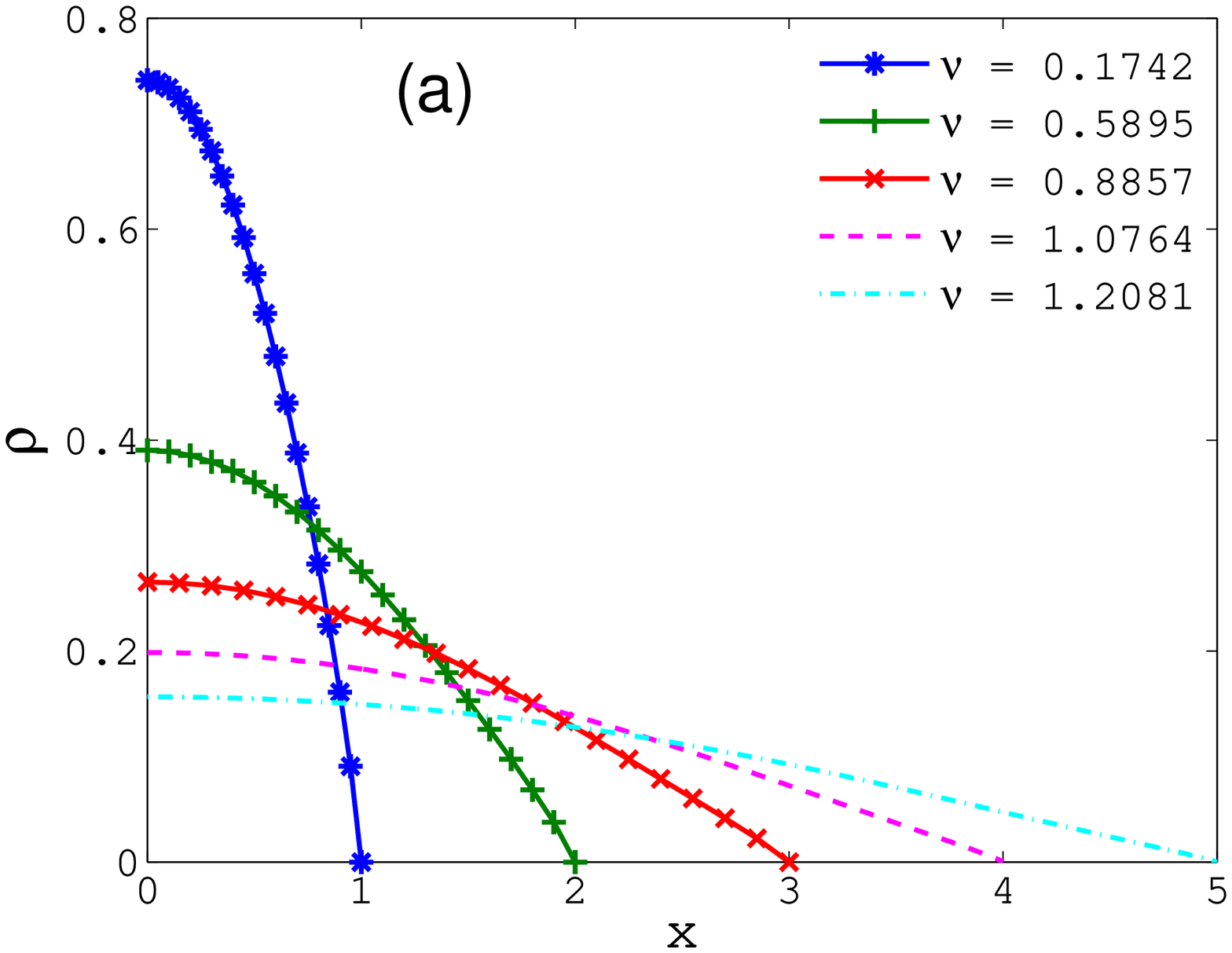}
$~~~$
  \includegraphics[totalheight=0.25\textheight]{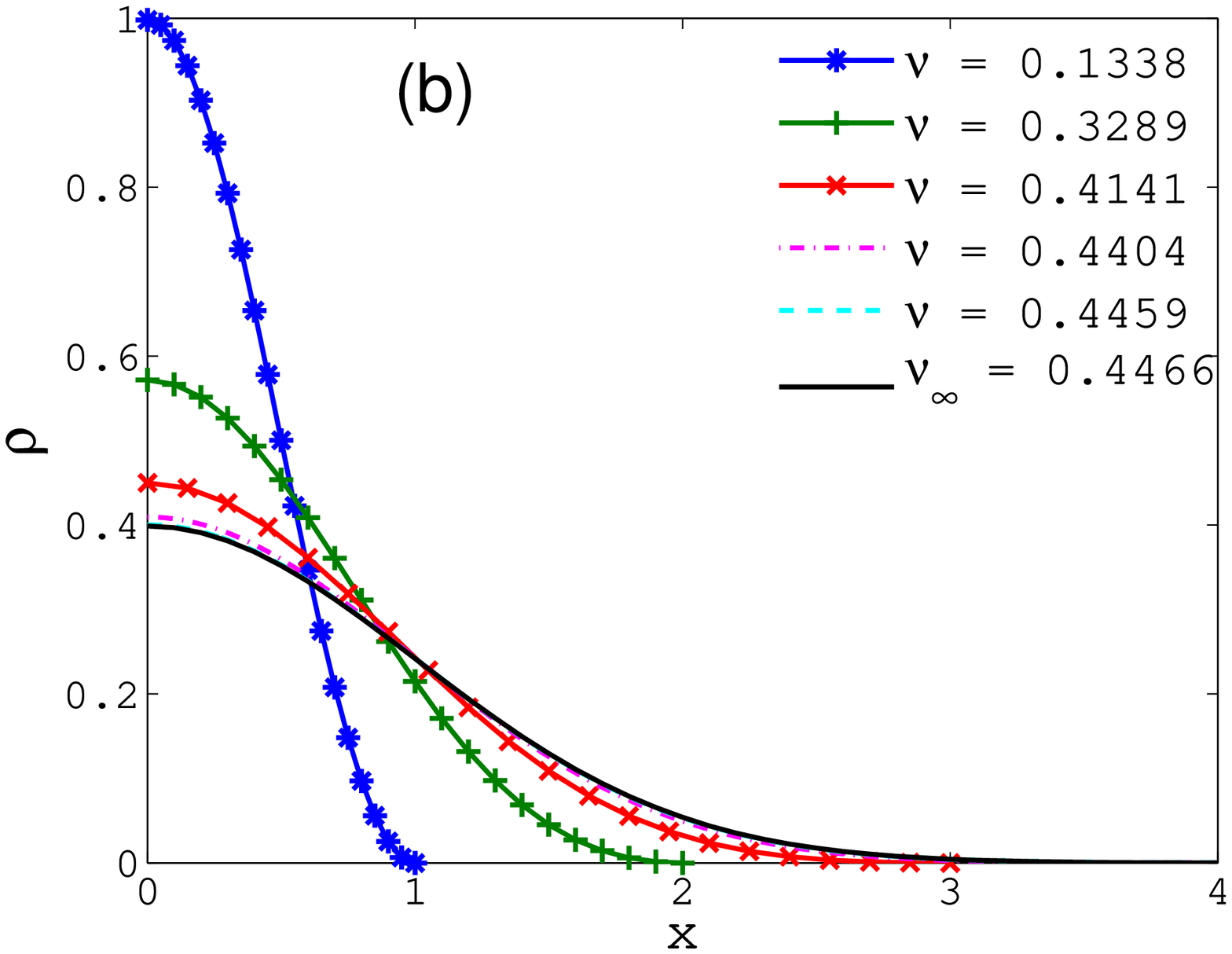}
 \end{center}
\caption{The steady states for the Gaussian 
kernel $G(x)=e^{-|x|^2/2}/\sqrt{2\pi}$ for ({\bf a}) 
$m=2.2$ and ({\bf b})
$m=1.5$, supported on domains of different sizes $L=1,2,3,4,5$ (the corresponding diffusion coefficients are shown in the legend). As $L$ increases to infinity the steady states approach a rescaled characteristic function when $m>2$, and converge to a fixed profile $\rho_\infty$ governed  by~\eqref{eq:nleig} when $m\in (1,2)$. In the latter case, the steady state with $\nu=0.4459$ (or $L=5$)
is indistinguishable from $\rho_\infty$ on the scale of the figure.}
\label{fig:sstate}
\end{figure}

\subsubsection*{Case $1<m <2$}
Solutions in this case (see Figure~\ref{fig:sstate}(b)) are also decreasing, but change concavity and reach  the boundary $L$ of the support with zero derivative (results consistent with analytical considerations of Section \ref{subsect:fp}). 

Also in contrast with the previous case, the steady states approach a fixed profile as $L$ goes to infinity. We can identify this limiting profile as follows. As $L\to \infty$, $\nu$ approaches a constant value $\nu_\infty$ (Figure \ref{fig:bifur}) and the constant $C$ in \eqref{eqn:ss} approaches zero (see discussion on equilibria of infinite support in Section \ref{sect:statsol}). Combing these observations, we infer from \eqref{eqn:ss} that  the limiting profiles $\rho_\infty$ are governed 
by the nonlinear eigenvalue problem 
\begin{equation}\label{eq:nleig}
\nu_\infty \rho_\infty^{m-1}(x) = G*\rho_\infty(x), \qquad x \in \mathbb{R}.
\end{equation}
In general, although no closed form expressions of the
eigenpair $(\nu_\infty,\rho_\infty)$ are expected, 
\eqref{eq:nleig} can be solved exactly for some special kernels $G$ like the Gaussian $e^{-|x|^2}/\sqrt{2\pi}$ or the Bessel potential $e^{-|x|}/2$.

%



We present the explicit calculation for the limiting profiles $\rho_\infty$ corresponding to the Gaussian kernel $G(x)=e^{-|x|^2}/\sqrt{2\pi}.$
For this kernel, the nonlinear eigenvalue problem~\eqref{eq:nleig} reads
\begin{equation}
\label{eqn:ssG-C0}
\nu_\infty \rho_\infty(x)^{m-1} = \frac{1}{\sqrt{2\pi}} \int_{-\infty}^{\infty} e^{-(x-y)^2/2} \rho_\infty(y) dy, \qquad x \in \mathbb{R}.
\end{equation}
We look for a solution in Gaussian form:
\[
\rho_\infty(x) = e^{-x^2/2\sigma^2}/\sqrt{2\pi \sigma^2},
\]
and by substituting into \eqref{eqn:ssG-C0}, we find 
\[
 \sigma^2 = \frac{m-1}{2-m},\qquad 
 \nu_\infty = (2\pi)^{\frac{m}{2}-1} (m-1)^{\frac{m-1}{2}} (2-m)^{\frac{2-m}{2}}.
\]
For $m=1.5$, $\nu_\infty\approx 0.4466$ and 
$\rho_\infty(x)=e^{-|x|^2/2}/\sqrt{2\pi}$, solution confirmed by numerics in  Figure \ref{fig:sstate}(b).

To conclude, when $m\in(1,2)$, solutions of \eqref{eqn:ss} exist only for $\nu < \nu_\infty$.  Numerics indicates that these solutions are unique, and simulations of the evolution equation \eqref{eq:aggeq} in Section \ref{sect:dynamics} suggest  that these equilibria are only {\em local} attractors  for the dynamics.  As $\nu$ approaches $\nu_\infty$ from below, the size $L$ of the support of solutions tends to infinity, and solutions approach a fixed density profile $\rho_\infty$. The pair $(\nu_\infty,\rho_\infty)$ solves \eqref{eq:nleig}. No compactly supported steady states exist for $\nu\geq \nu_\infty$. In such a case the evolution of \eqref{eq:aggeq} is dominated by diffusion and spreading to a trivial solution occurs.


\subsection{Alternative Approach for the Bessel Potential $G(x)=e^{-|x|}/2$}
\label{sect:Morse}
For the Bessel potential $G(x)=e^{-|x|}/2$, in addition to using the 
general iterative scheme, various explicit calculations can be worked out by taking advantage of the fact that $G$ is the Green's function of the differential operator $\operatorname{Id}-\partial_{xx}$.  This special kernel has been already used in the study of steady states for $m=3$ in~\cite{TBL}, as well as in various other nonlocal aggregation equations, including  higher dimensional models \cite{BeTo2011,LeToBe2009}. This alternative approach confirms that the qualitative behaviours for the Gaussian potential apply also to the Bessel potential, suggesting that the results are generic to all potentials in the class considered in this paper.

To this end, by symmetry of the solution, the equation~\eqref{eqn:ss}
reads
\begin{equation}
\label{eqn:ss-exph}
\nu \rho(x)^{m-1} = 
\frac{1}{2}\int_0^L (e^{-|x-y|}+e^{-|x+y|})\rho(y)dy-C, \qquad x\in [0,L],
\end{equation}
where 
\begin{equation}
\label{eqn:C-exp}
C = \frac{1}{2}\int_0^L (e^{-L-y}+e^{-L+y})
\rho(y)dy.
\end{equation}

Use the change of variable  $p(x)=\rho(x)^{m-1}$ and apply the
differential operator $\operatorname{Id}-\partial_{xx}$ to \eqref{eqn:ss-exph}, to get 
\begin{equation}
\label{eqn:diff-v}
 p-\frac{d^2p}{dx^2} = \frac{1}{\nu}\left(p^{\frac{1}{m-1}}-C\right),\qquad  x\in [0,L],
\end{equation}
with the boundary conditions $p'(0)=0$ and $p(L)=0$. Note that we assume here that solutions have the properties considered in Section \ref{subsect:fp}.

Multiplying both sides of  \eqref{eqn:diff-v} with $p_x$ and integrating with respect to $x$, one finds
\begin{equation}
\label{eq:diff-intv}
 \frac{1}{2}p^2 -\frac{1}{2}p_x^2 = \frac{m-1}{m\nu}p^{\frac{m}{m-1}} - \frac{C}{\nu}p+C_1.
\end{equation}
Here the integration constant $C_1$ can be obtained from the conditions at the origin $p'(0)=0$, i.e.,
\[
 C_1 = \frac{1}{2}p(0)^2 -\frac{m-1}{m\nu}p(0)^{\frac{m}{m-1}}
+\frac{C}{\nu}p(0).
\]
Solving $p'(x)$ from~\eqref{eq:diff-intv} using the fact that
$p'(x)\leq 0$ on $[0,L]$, the solution can now be expressed
implicitly as
\begin{equation}
\label{eqn:impl}
x = \int_{p(x)}^{p(0)} \left[
p^2-p(0)^2-\frac{2(m-1)}{m\nu}\big(p^{m/(m-1)}-p(0)^{m/(m-1)}\big)+
\frac{2C}{\nu}\big(p-p(0)\big)
\right]^{-1/2}dp,
\end{equation}
where  the size $L$ of the support is found from the condition $p(L)=0$:
\begin{equation}
\label{eqn:L}
L = \int_0^{p(0)}  \left[p^2-p(0)^2-\frac{2(m-1)}{m\nu}(p^{m/(m-1)}-p(0)^{m/(m-1)})+\frac{2C}{\nu}(p-p(0))
\right]^{-1/2}dp.
\end{equation}

For fixed $\nu$, the problem of searching for the steady states is
now reduced to finding the parameters $p(0)$ and $C$ such that the solution $\rho =
p^{1/(m-1)}$ (depending on $p(0)$ and $C$) satisfies the
constraints of unit mass and equation~\eqref{eqn:C-exp} for $C$.

This problem can be further simplified if instead of the unit mass condition we use the normalization
$\bar{\rho}(0)=\bar{p}(0)=1$ for the solution. As explained in Section~\ref{sect:ss-pf}, a simple rescaling in the magnitude of the
solution can be used to go between solutions with different types of normalizations. Hence, let $\bar{p}(x)$ be the solution
defined implicitly by~\eqref{eqn:impl} (with $C$ replaced by $\bar{C}$) which satisfies $\bar{p}(0)=1$.  By defining the function 
$$
I(\bar{C}) := \frac{1}{2}\int_0^L 
(e^{-L-y}+e^{-L+y})\bar{p}(y)^{1/(m-1)}dy,
$$
we infer from \eqref{eqn:C-exp} that $\bar{C}$ is a solution of the fixed point problem $I(\bar{C})=\bar{C}$. 
Notice that passing to the renormalization to unit mass, $\nu$ and 
$C$ have to be scaled along with $\rho$, i.e.,
\begin{equation}\label{eq:rescale}
    \rho \to \frac{\bar{\rho}}{\int_{-L}^L \bar{\rho}(y) dy}
    ,\quad
    \nu \to \bar{\nu} \Big(\int_{-L}^L \bar{\rho}(y)dy\Big)^{m-2},
    \quad 
    C \to \bar{C}\int_{-L}^L \bar{\rho}(y)dy,
\end{equation}
with $\bar{\rho}=\bar{p}^{1/(m-1)}$. 

When $\bar{\nu}$ is fixed and $\bar{p}(0)=1$, the interval on which 
we search for the fixed point of $I(\bar{C})$ is restricted  by the requirement $\bar{p}''(0)\leq 0$  in \eqref{eqn:diff-v} and by the non-negativity of the expression inside the square bracket in~\eqref{eqn:impl} at $p=0$. Both restrictions 
lead to upper bounds on $\bar{C}$, and can be written as 
\begin{equation}\label{eq:Cbound}
 \bar{C} \leq \min\left(1-\bar{\nu}, \frac{m-1}{m}-\frac{\bar{\nu}}{2}\right).
\end{equation}

The function $I(\bar{C})$, shown in Figure~\ref{fig:besselfixed}(a)
for $m=3$, has a fixed point only for $\bar{\nu}$ small
enough (as indicated by~\eqref{eq:Cbound}). An important observation is that the fixed point appears to be {\em unique}, consolidating our observation regarding unique solutions of \eqref{eqn:ss}.

\begin{figure}[htp]
    \begin{center}
        \includegraphics[totalheight=0.25\textheight]{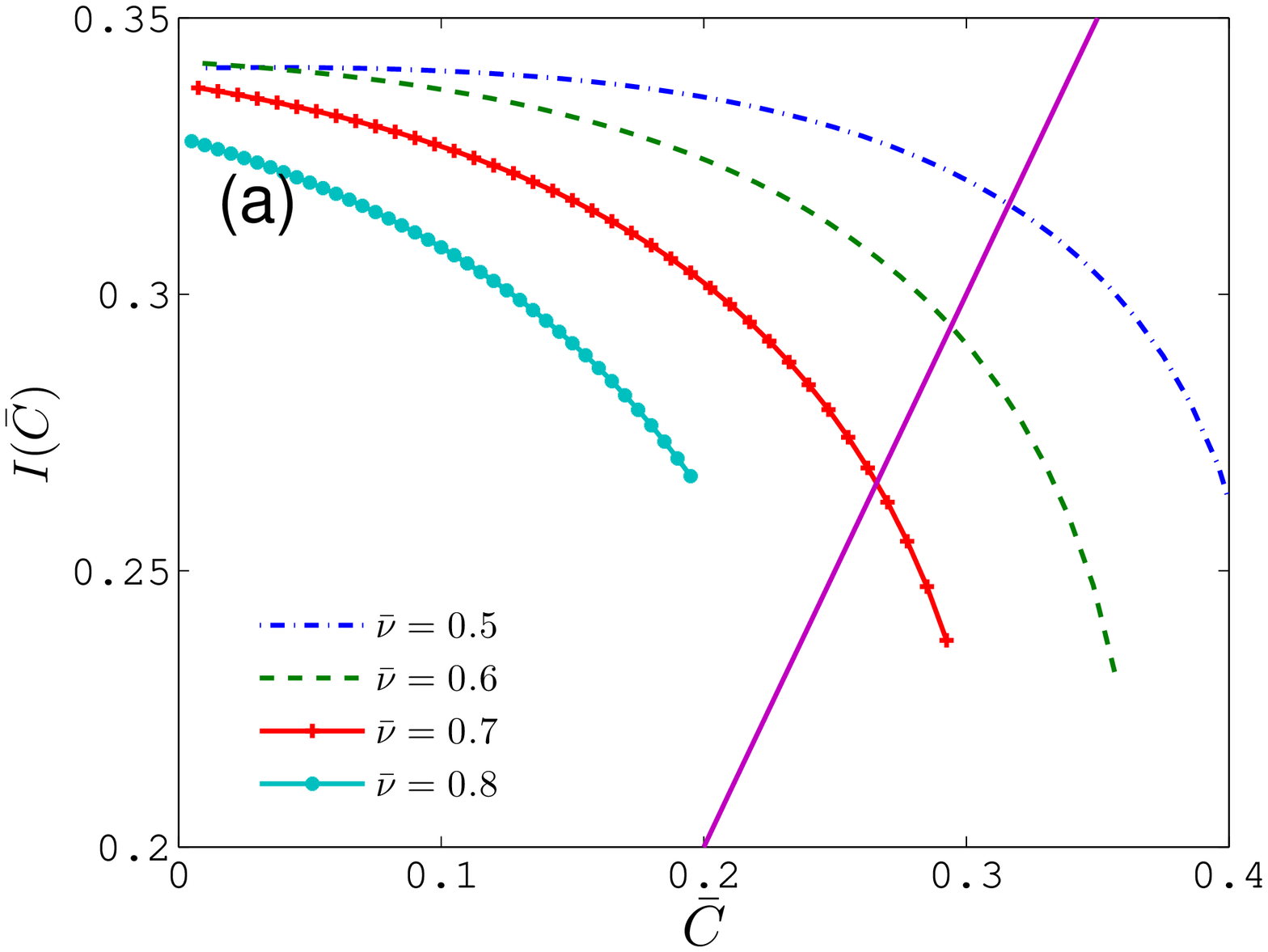}
$~~~$
        \includegraphics[totalheight=0.25\textheight]{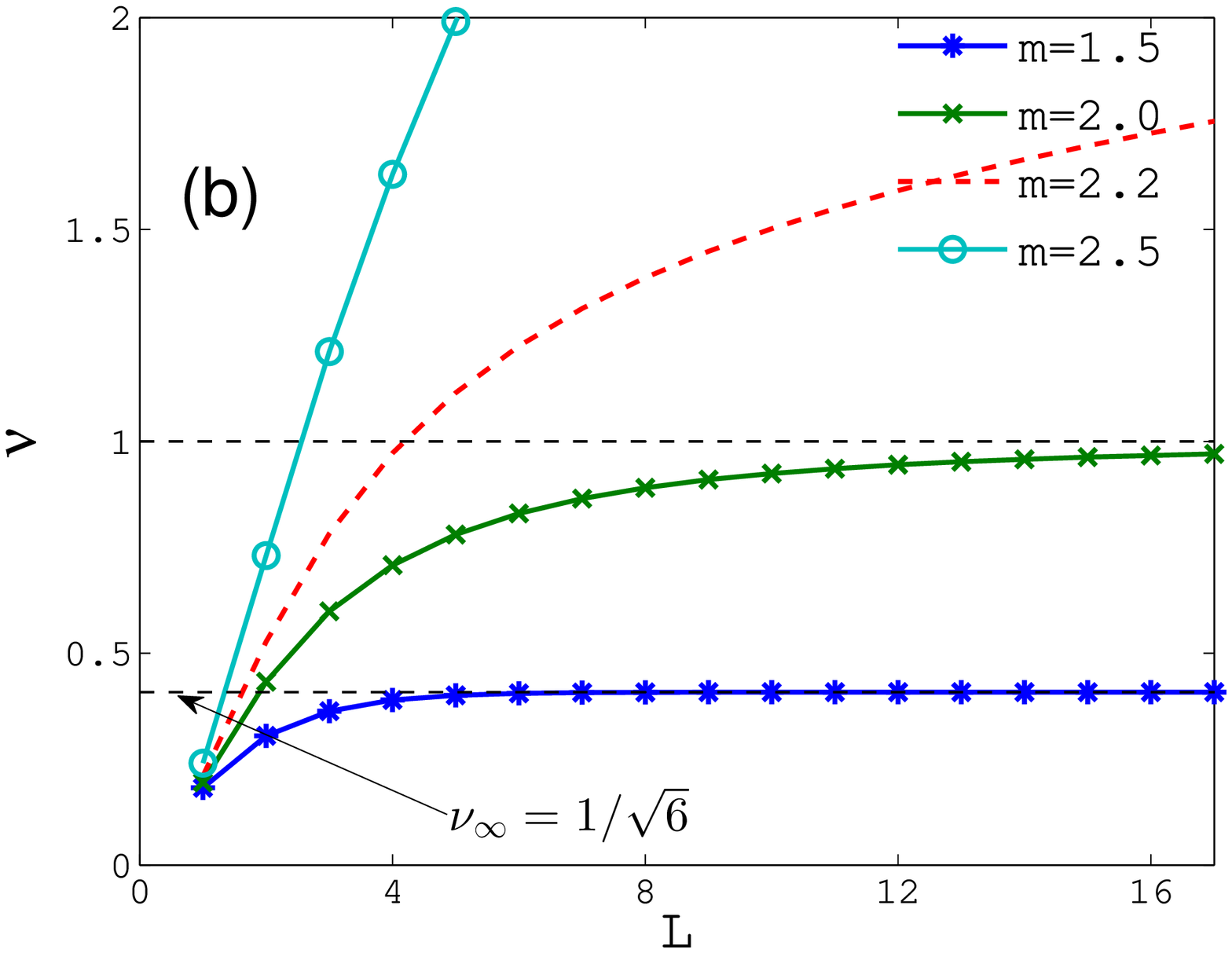}
    \end{center}
    \caption{ {\bf (a)} The intersection of $I(\bar{C})$ with $\bar{C}$
    for $m=3$ and different $\bar{\nu}$. {\bf (b)} The dependence
    of $\nu$ on $L$ by finding the fixed point of $I(\bar{C})$ 
and rescaling to the unit mass normalization.}
    \label{fig:besselfixed}
\end{figure}

Rescaling the variables using the unit mass
normalization~\eqref{eq:rescale}, the dependence of $\nu$
on $L$ is shown in Figure~\ref{fig:besselfixed}(b). The result is similar to Figure~\ref{fig:bifur}, corresponding to the Gaussian kernel, and its interpretation follows closely the considerations made there.

The steady states computed from \eqref{eqn:diff-v} are also similar
to the corresponding equilibria for the Gaussian kernel (Figure~\ref{fig:sstate}),
and are not shown here. When $m>2$, as $\nu$ increases to infinity, 
the convergence of the steady state to a rescaled 
characteristic function can be explained from 
a phase-plane analysis of the  ODE~\eqref{eqn:diff-v} or the implicit equation~\eqref{eqn:impl}. It can be shown (details not included) that in this limit, the coefficient $\nu$ and the constant $C$ relate in such a way that the solution $p(x)$ stays indeed very close to $p(0)$ for 
$x$ far away from the origin. 

When $m\in (1,2)$, the limiting profiles $\rho_\infty$ can also be
obtained explicitly, with $\nu_\infty$ solved from an algebraic
equation. Using the vanishing condition
$C=0$ and $\rho_\infty(\infty)\to 0$, the equation for $p_\infty =
\rho_\infty^{m-1}$, when integrated once as in~\eqref{eq:diff-intv}, 
becomes
\begin{equation}
    \frac{1}{2}p_\infty^2 -\frac{1}{2}p_\infty'^2 =
    \frac{m-1}{m\nu_\infty} p_\infty^{\frac{m}{m-1}}.
\end{equation}
In particular, $p_\infty 
(0)=\big( \frac{m\nu_\infty}{2(m-1)}\big)^{(m-1)/(2-m)}$.
As opposed to the general implicit equation~\eqref{eqn:impl}, this equation can be integrated explicitly as 
\[
    x = \int_{p_\infty(x)}^{p_\infty(0)}
    p^{-1}\left[1-\left(\frac{p}{p_\infty(0)}
        \right)^{\frac{2-m}{m-1}}\right]^{-1/2} dp =
        \frac{2(m-1)}{2-m}\mbox{arctanh}\sqrt{
        1- \left(\frac{p_\infty(x)}{p_\infty(0)}
        \right)^{\frac{2-m}{m-1}}}.
\]
Therefore $\rho_\infty $ is then given by 
\[
    \rho_\infty(x) = \left(\frac{m\nu_\infty}{2(m-1)}\right)^{1/(2-m)} 
    \left[1-\tanh^2 \frac{2-m}{2(m-1)}x\right]^{1/(2-m)},
\]
where the constant $\nu_\infty$ is determined from the unit 
total mass, $\int_{\mathbb{R}} \rho_\infty(x)dx = 1$. The exact 
value of $\nu_\infty$ can be obtained in a few cases, for instance,
$\nu_\infty = 1/\sqrt{6}$ when $m=3/2$
(see Figure~\ref{fig:besselfixed}(b)) and $\nu_\infty =
(2\pi^2)^{-1/3}$ when $m=4/3$.



\paragraph{Remark.} In higher dimensions, similar techniques involving the
shooting method and a fixed point equation can be
formulated to find the steady states, however the analogous equation to~\eqref{eq:ss3}
for the derivative of the density does not seem to exist. 
As a result, there is no such iterative scheme as in section~\ref{subsect:intnum} to compute the steady states. This is the primary reason we focus on one dimension in this paper,  although we expect similar qualitatively behaviours and demarcations at $m=2$. 
Among general kernels, the Bessel potentials are the very few cases when the corresponding steady states can be calculated without solving the evolution equation. The situation is however much more delicate in higher dimensions. It can be shown that, due to radial symmetry, the corresponding integral equation \eqref{eqn:ss-hd-gen} in higher dimensions can be converted into the ODE
\begin{equation}
 p-\frac{n-1}{r}\frac{dp}{dr}-\frac{d^2p}{dr^2} = \frac{1}{\nu}
\big(p^{\frac{1}{m-1}}-C\big).
\end{equation}
Even though $m=2$ is still expected to be a critical exponent,
there is no solution when $m$ is smaller than $m^*=2n/(n+2)$ or
$1/(m-1)$ is larger than the critical Sobolev exponent
$2^*-1=(n-2)/(n+2)$ in dimensions greater than two, at least for the limiting case~\cite{MR695535} with $C=0$. Another complication 
is the singularity at the origin, which leads to blowup solutions
when $m<2-2/n$ and the diffusion is not strong enough to 
balance the aggregation~\cite{KozonoSugiyama08}.




\section{Dynamic Evolution}
\label{sect:dynamics}

In this section we compute numerically the solutions to equation \eqref{eq:aggeq} to show that the steady states analyzed in the previous sections capture indeed the long time behaviour of the aggregation model.

\subsection{Numerical Method}
\label{subsect:dynnum}
We use the numerical method recently developed  in \cite{CCH} to deal specifically with aggregation equations like \eqref{eq:aggeq}, which contain interaction terms, nonlinear diffusion, and have a gradient flow structure. The method is based on a finite-volume scheme that preserves positivity and has the desired energy dissipation properties. We present briefly the method and for details we refer to \cite{CCH}.

We take a computational domain $[-L,L]$ with equally spaced grid points $-L=x_0<x_1<\cdots<x_N=L$. Consider the midpoints $x_{j-1/2}=(x_{j-1}+x_j)/2$ and define $\bar{\rho}_j$ to be the average of the density on the cell $C_j=[x_{j-1/2},x_{j+1/2}]$. The finite volume method from \cite{CCH} consists in a semi-discrete scheme in conservative form for $\bar{\rho}_j$:
\begin{equation}\label{eq:semi-scheme}
 \frac{d}{dt}\bar{\rho}_j(t) = -\frac{F_{j+1/2}(t)-F_{j-1/2}(t)}{
\Delta x},
\end{equation}
where $F_{j+1/2}$ approximates the continuous flux $\rho(G*\rho-\nu \rho^{m-1})_x$ 
at cell interfaces $x_{j+1/2}$. 

The numerical fluxes $F_{j+1/2}$ are computed as follows. Consider the numerical approximations of the velocities at $x_{j+1/2}$:
$$ u_{j+1/2} = (\xi_{j+1}-\xi_j)/\Delta x, $$where
$$ \xi_j = \Delta x\sum_k G(x_j-x_k)\bar{\rho}_k-\nu \bar{\rho}_j^{m-1}.$$
Then, the numerical flux $F_{j+1/2}$ is approximated  by
\[
 F_{j+1/2} = u_{j+1/2}^+\rho_{j+1/2}^+ + u_{j+1/2}^-\rho_{j+1/2}^-,
\]
where $u_{j+1/2}^+=\max(u_{j+1/2},0), u_{j+1/2}^-=\min(u_{j+1/2},0)$
and $\rho_{j+1/2}^\pm$ is the one-sided density at $x_{j+1/2}$ 
which depends on the spatial order of the scheme: $\rho_{j+1/2}^+=\bar{\rho}_{j+1},
\rho_{j+1/2}^-=\bar{\rho}_j$ for first order and 
$\rho_{j+1/2}^\pm$ reconstructed from minmod or other 
slope-limiter for second order. 

In all numerical examples  the second order finite-volume method  is used, while the system of \eqref{eq:semi-scheme} is integrated using the third-order strong-stability preserving Runge-Kutta method~\cite{GST}. Despite its simplicity, this scheme preserves the positivity
of the solution and dissipates the energy (see~\cite{CCH} for more details and extensive examples), which is  critical for the study of long time behaviour of the solutions.


\subsection{Asymptotic Behaviour and Metastability ($m>2$)}
\label{sec:coarsemeta}
We evolve solutions to \eqref{eq:aggeq} in time and observe how they approach equilibria.   To allow for a broad range of data, the initial densities are generated randomly as follows.  First we generate density values on a coarse grid from the uniform distribution on the interval [0,1]. Then multiply everywhere with a Gaussian function of a certain width to ensure sufficient decay at the end of the intervals. Finally interpolate the density values on a fine grid and normalize to unit mass.

\begin{figure}[htb]
\begin{center} 
\includegraphics[totalheight=0.25\textheight]{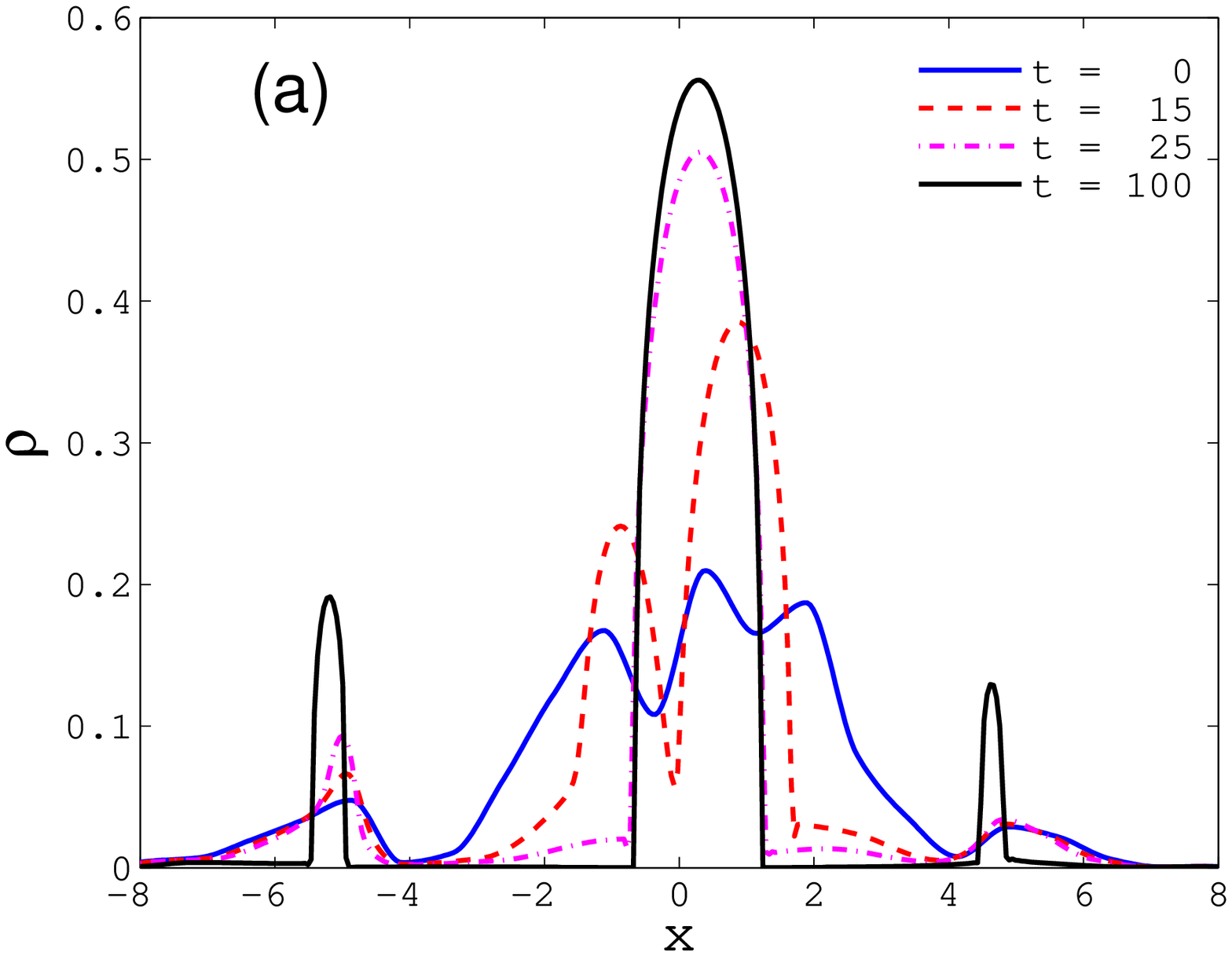} 
$ ~~~~ $
\includegraphics[totalheight=0.25\textheight]{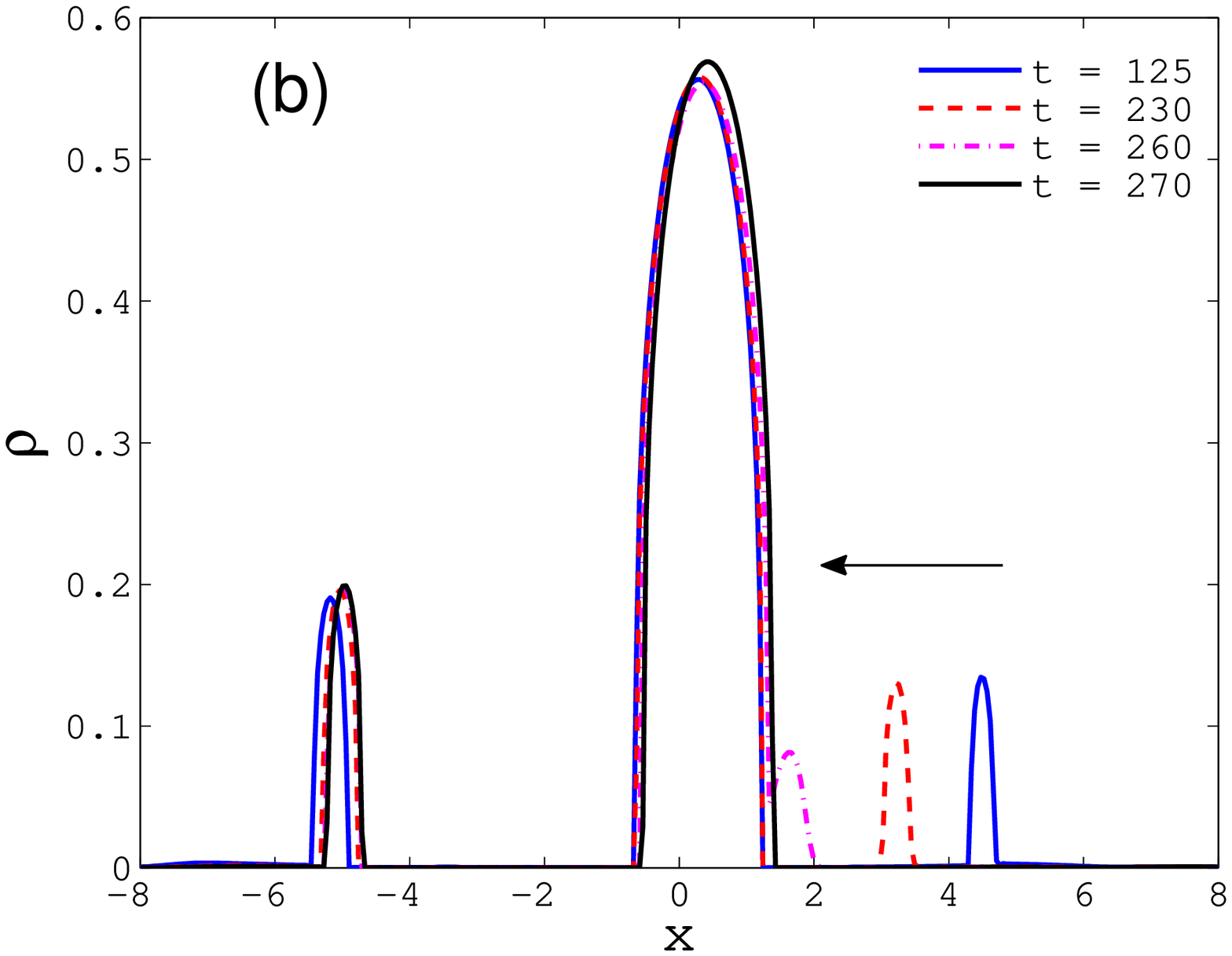} \\
\includegraphics[totalheight=0.25\textheight]{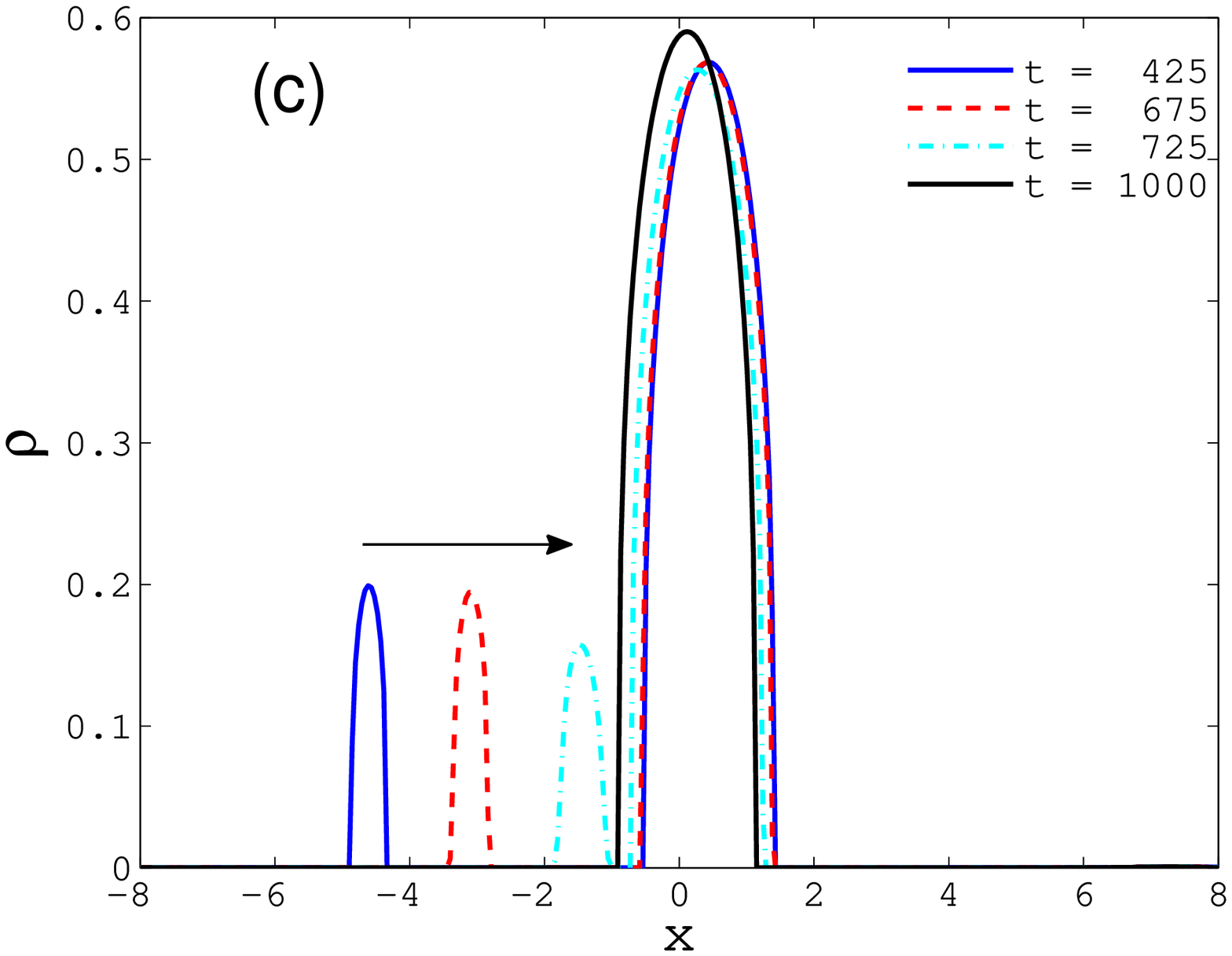}
$ ~~~~ $
\includegraphics[totalheight=0.24\textheight]{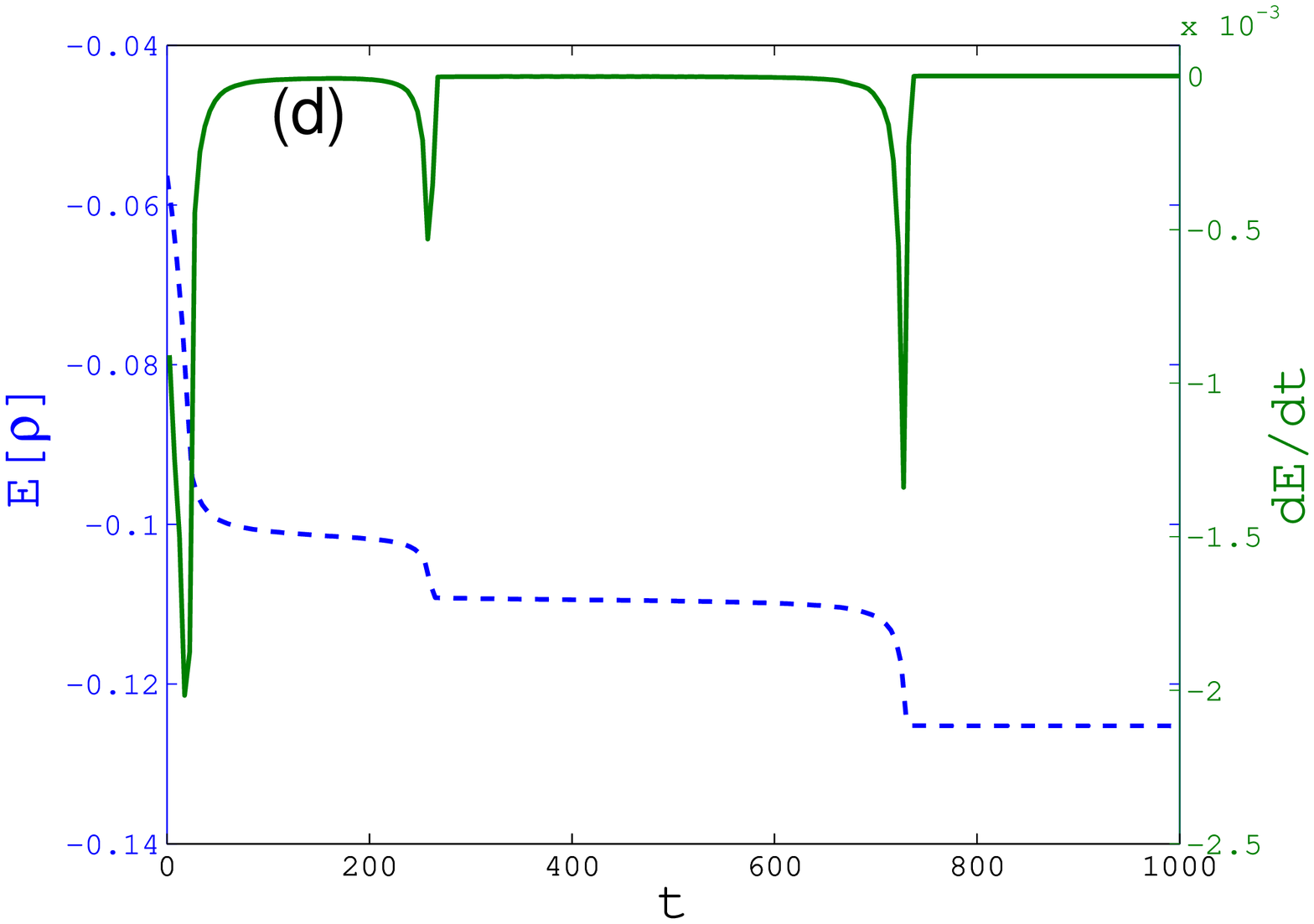}  
\end{center}
\caption{Dynamic evolution of  \eqref{eq:aggeq} with $m=4$, $\nu=0.6$, and a Bessel attractive potential.  Plot ({\bf a}) demonstrates the coarsening of a randomly generated initial density. Plots ({\bf b}) and ({\bf c}) show the subsequent time evolution, with the smaller clumps eventually merging into the larger one. Plot ({\bf d}) shows a staircase-like evolution of the energy (dashed line), along with the plot of its time derivative (solid line). The level regions of the energy correspond to metastable states consisting of multiple groups, while the sharp drops in $E$ and the peaks in its derivative $dE/dt$ correspond to group mergers.}
 \label{fig:dynamics-mg2}
\end{figure}

\begin{figure}
 \begin{center}
  \includegraphics[totalheight=0.21\textheight]{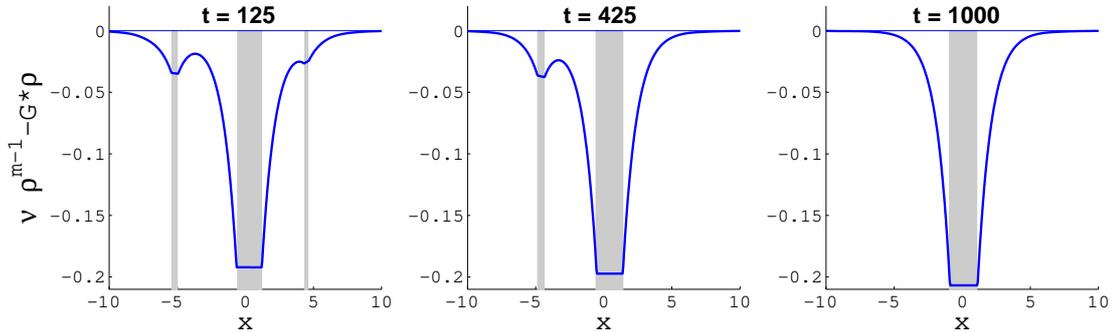}
 \end{center}
\caption{The plot of $\delta E/\delta \rho = \nu \rho^{m-1}-G*\rho$ at $t=125$, $t=425$ and $t=1000$ --- see Figure \ref{fig:dynamics-mg2}.  The graphs show  that $\delta E/\delta \rho$ is constant on the different 
components of $\mathrm{supp}[\rho]$ (shaded regions). The (relatively weak) interaction between components is through the convolution $G*\rho$. 
}
\label{fig:metaxi}
\end{figure}

Figure \ref{fig:dynamics-mg2}(a)-(c) shows the time evolution of the solution of \eqref{eq:aggeq} for $m=4$, $\nu =0.6$ starting from such a randomly generated initial density ($t=0$ shown in plot Figure 3(a)). The attractive interaction kernel is the Bessel potential $G(x)=e^{-|x|}/2$. There is a fast coarsening of the initial density distribution resulting in three clumps (more clumps may appear for initial data with larger width), which eventually merge on a slow time scale. 

The fast coarsening followed by slow group merging is characteristic to all simulations we performed, but it is more pronounced in the aggregation dominated regime of large $m$ 
(when $\rho(x,t)<1$) and small $\nu$.
 Our results are also consistent with numerical observations made in \cite{TBL} using cubic nonlinearity ($m=3$). 

To further understand this coarsening and metastability,  it is very important to monitor dynamically the energy given by \eqref{def:entropyfunctional}, shown in Figure \ref{fig:dynamics-mg2}(d). Note that after the initial coarsening, it decreases in a staircase fashion. A flat region corresponds to a multiple-clump configuration, which acts as a metastable state where the solution could spend a considerable amount of time. We chose to show this particular numerical simulation where the dynamics escapes the multiple-clump states rather fast, but we have seen cases where a multiple-clump state persists for a very long time. In fact, this delicate dynamics near a metastable state required the design of a suitable numerical scheme. Since the attractive kernels we use decay relatively fast with distance, once two groups get separated by a certain distance, they interact with each other very weakly and a long time may pass before their merger occurs. The sharp drops in the energy correspond to clump mergers. The 
closer clump on the right first merges with the large one in the center at around $t=265$. After the first merger, the two remaining clumps will persist for some time in a metastable configuration, before the left clump starts travelling to the right and merge with the central group. The second merger occurs at around $t=730$, corresponding to the second significant drop in energy. After that, the dynamics settles into the steady state identified and investigated in Sections  \ref{sect:ss-pf} and \ref{sect:Morse}.  

One implication from these results is that there is no entropy-entropy dissipation
inequality of the form 
\[
 \frac{d}{dt} \big( E[\rho(\cdot,t)]-E[\rho(\cdot,\infty)]\big) \leq - 
\mu\big( E[\rho(\cdot,t)]-E[\rho(\cdot,\infty)]\big),
\]
and the convergence of the solution $\rho(\cdot,t)$ to the steady state $\rho(\cdot,\infty)$ can be arbitrarily slow, in contrast to the fast convergence 
 for other equations like the rescaled
porous medium equation~\cite{carrillo2000asymptotic} or the Keller-Segel equation with subcritical mass~\cite{calvez2012refined}.

The metastability can be understood from the plots of $\delta E/\delta \rho = \nu \rho^{m-1} - G*\rho$ shown in  Figure~\ref{fig:metaxi}. It is easy to see that, when the clumps are away from 
each other, $\delta E/\delta \rho$ is almost equal to a constant on each such component of the support of $\rho$ (shaded area), 
or $\partial_x( \delta E/\delta \rho) \approx 0$. 
Therefore, $\frac{d}{dt} E = -\int \rho |\partial_x( \delta E/\delta \rho)|^2 
\approx 0$, and this explains the level regions observed
 in Figure~\ref{fig:dynamics-mg2}(d). When the clumps are far away from each other, the velocity of one particular clump is due to the nonlocal interaction 
kernel $G$, and decays as a function of distance between the other clumps 
in the same rate as $G$~\cite{slepcev2008coarsening}.  
In the extreme case of a compactly supported kernel $G$, steady states with multiple disconnected bumps do exist, where each such component is a locally stable  steady state.

Coarsening, as well as metastability, are well-known phenomena in phase separation models \cite{BatesXun94}, but are mainly studied in local systems with non-convex energy. Formation of metastable states in the case of nonlocal attraction has also been observed in chemotactic models \cite{DolakSchmeiser05,MR2397995}. In these works the metastability was created by a doubly degenerate mobility, and indeed, in the zero diffusion limit, the metastable clusters became stable stationary solutions. In the case of the model we consider here, the metastability seems to be the result of a quasi-stationary  energetic balance of attraction and repulsion,  rather than being induced by mobility. A further analytical study of such effect seems to be a valuable subject for future research. 

We conclude this subsection with a conjecture. Based on the numerical experiments and the theoretical results we conjecture that for $m>2$ the equilibria studied in Sections  \ref{sect:ss-pf} and \ref{sec:numss} are global attractors for the dynamics of \eqref{eq:aggeq}. Note that in this regime a global minimizer of the energy \eqref{def:entropyfunctional}-\eqref{eqn:powerfF} is known to exist \cite{Bedrossian}.


\subsection{Limited Basin of Attraction ($1<m<2$)}

While the coarsening of the initial data and the metastability of the dynamics is still observed in this regime, the stronger diffusion (at least when the maximum density is less than one) 
leads to richer long time asymptotic behaviours. More precisely, the equilibria of Sections \ref{sect:ss-pf} and \ref{sec:numss} are approached only when the diffusion coefficient is small enough  ($\nu<\nu_\infty$) and the initial data is localized near its center of mass.

\begin{figure}[htp]
 \begin{center}
  \includegraphics[totalheight=0.26\textheight]{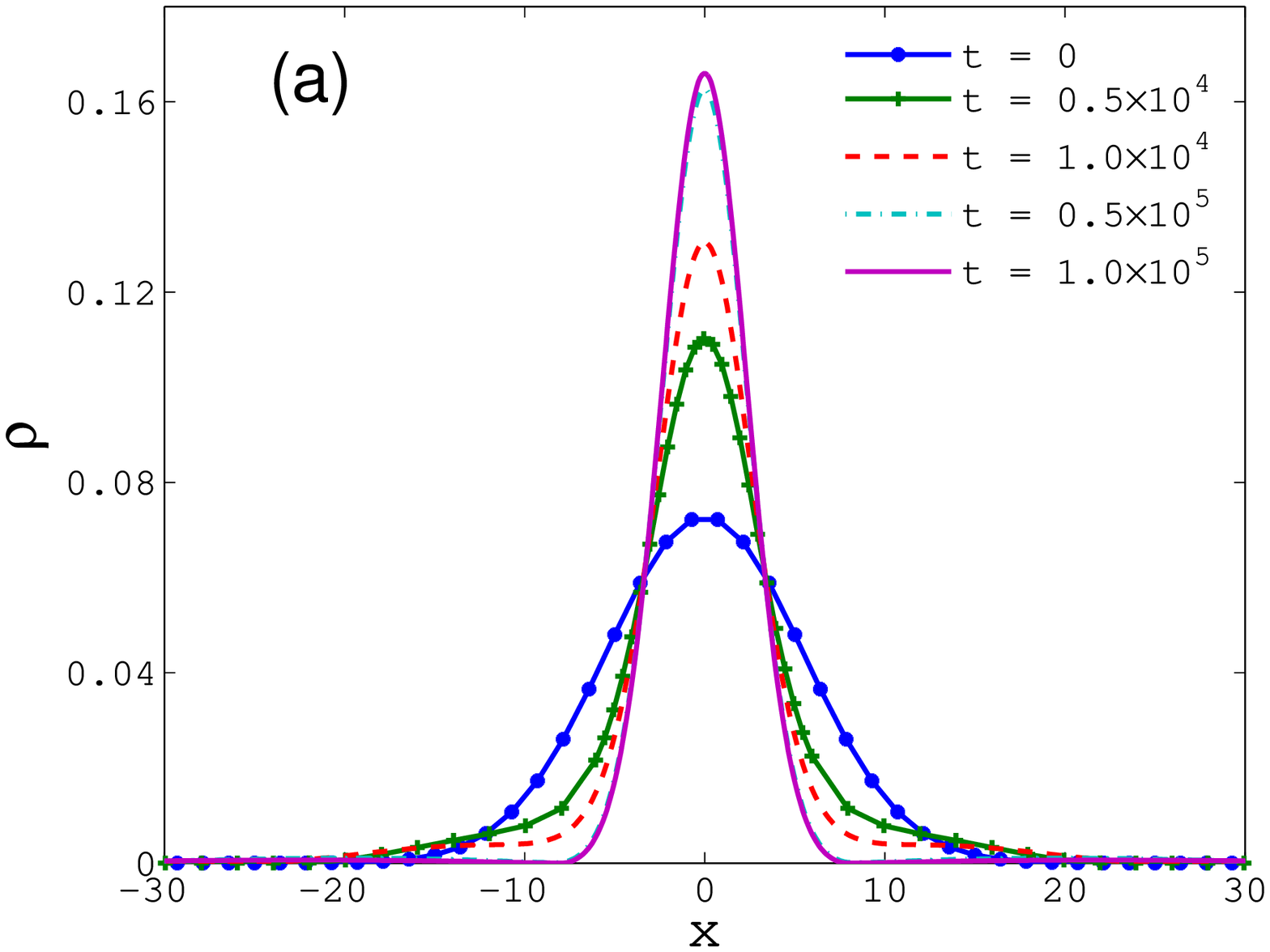}
$~~$
\includegraphics[totalheight=0.25\textheight]{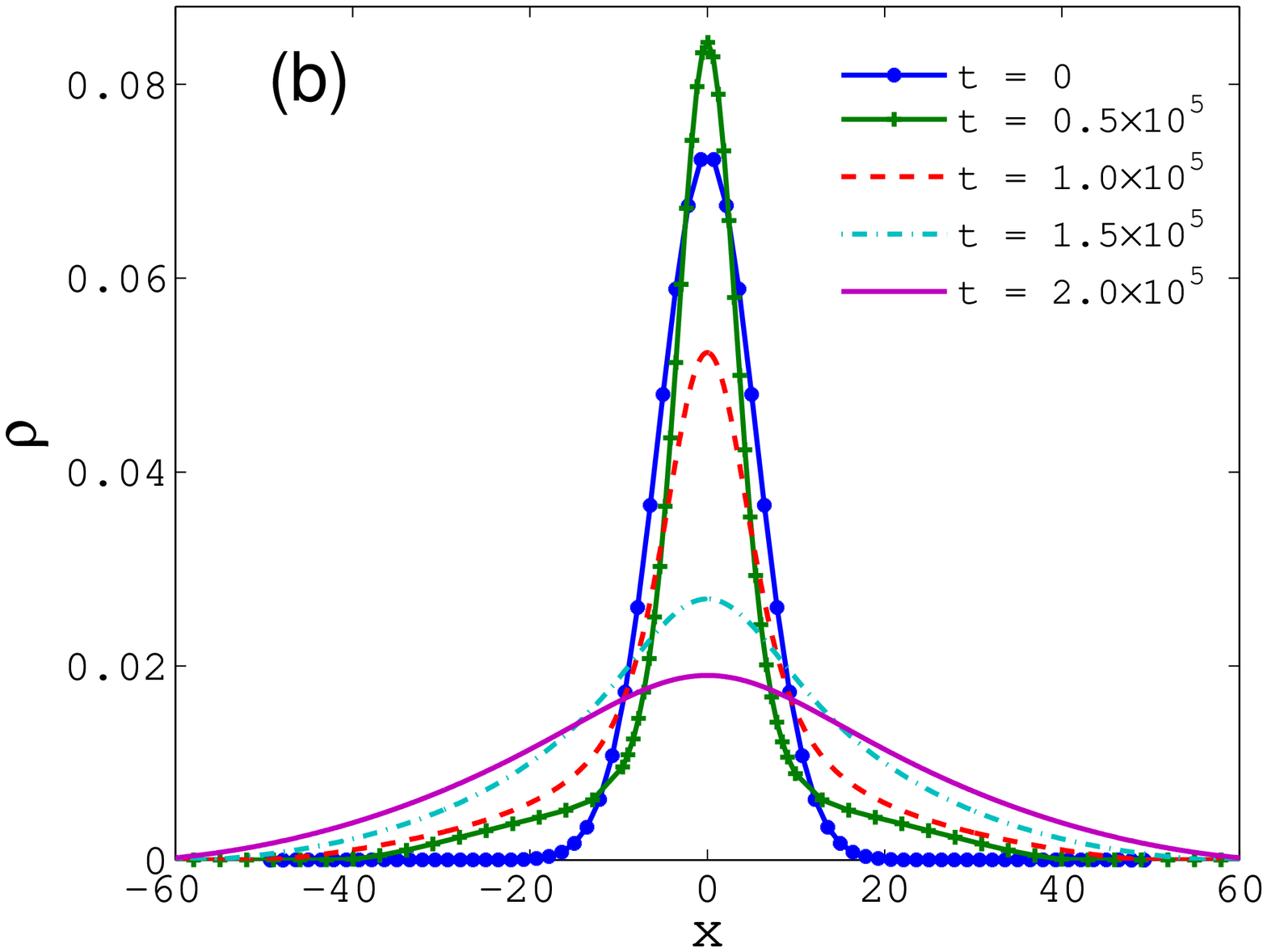}
 \end{center}
\caption{Different long time behaviours of the solution for $m=1.8$. 
Here $G(x)=e^{-|x|^2/2}/\sqrt{2\pi}$, $\nu = 0.60<\nu_\infty\approx 0.6479$ 
and the initial data is a Gaussian $u_0(x) = 
e^{-|x|^2/2\sigma^2}/\sqrt{2\pi \sigma^2}$ ($\sigma^2=30$ in ({\bf a}) and $\sigma^2=50$ in ({\bf b})). The solution 
converges either to the compactly supported steady state or to the trivial solution, depending on the spread of the initial profile.}
\label{fig:smallmevol}
\end{figure}

The dynamic evolution for two Gaussian initial profiles with different widths for $m=1.8$, $ \nu = 0.60$ and $G(x)=e^{-|x|^2}/\sqrt{2\pi}$ is shown in Figure~\ref{fig:smallmevol}. 
Although the compactly supported steady state exists ($\nu < \nu_\infty\approx0.6479$), 
the solution converges to this stationary solution only when the 
width of the initial data is small enough; otherwise the solution spreads and decays to a trivial state. The steady states of Sections  \ref{sect:ss-pf} and \ref{sec:numss} are only local attractors for the dynamics. The decaying 
solution with large width can be explained formally by the long wave
approximation $G*\rho(x) \approx \|G\|_{L^1}\rho(x)$, where the 
original equation~\eqref{eq:aggeq} becomes 
\begin{equation}\label{eq:longwave}
 \partial_t \rho 
=\partial_x\big(\rho \partial_x(\nu \rho^{m-1}-\rho\|G\|_{L^1})\big).
\end{equation}
For small densities, the diffusion arising from $\nu \rho^{m-1}$ 
dominates the anti-diffusion from $\rho \|G\|_{L^1}$, and
the solution continues to decay to zero. 

\begin{figure}[htp]
 \begin{center}
\includegraphics[totalheight=0.26\textheight]{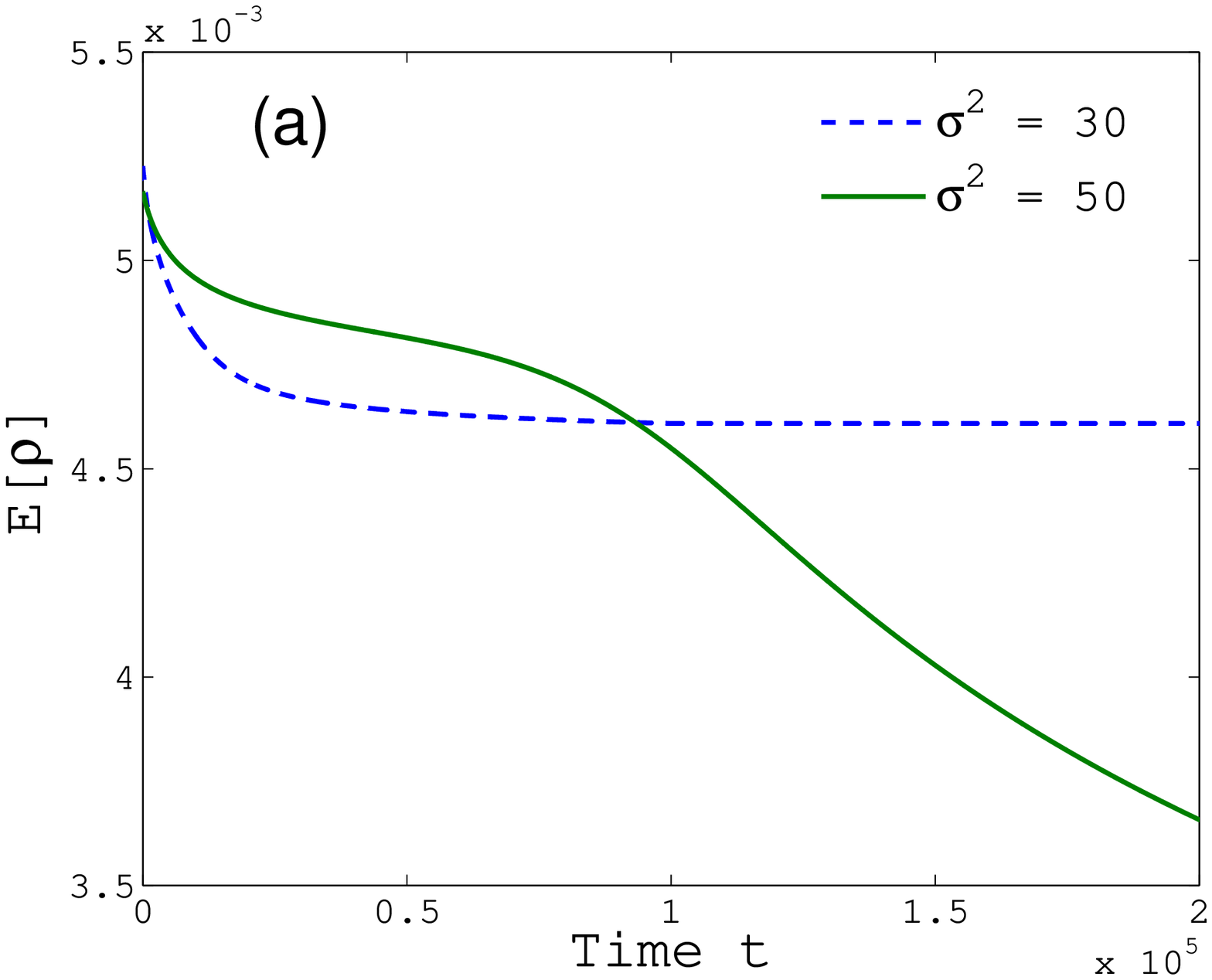}
$~~$
\includegraphics[totalheight=0.25\textheight]{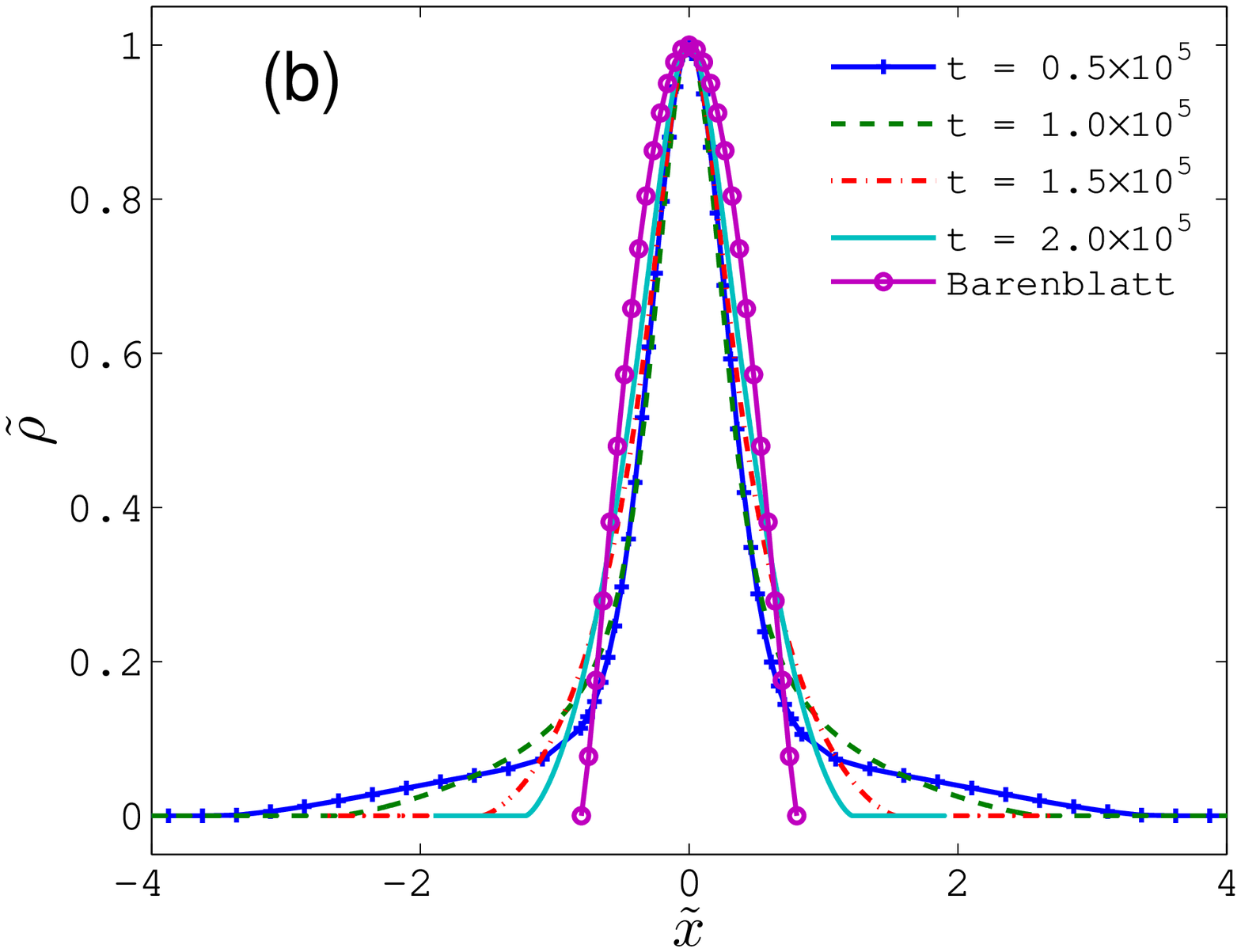}
 \end{center}
\caption{({\bf a}) Time decay of the energy corresponding to the simulations from Figure~\ref{fig:smallmevol}. ({\bf b}) Rescaled profiles of the spreading solution in  Figure~\ref{fig:smallmevol}(b) approach asymptotically a rescaled Barenblatt profile.}
\label{fig:engbaren}
\end{figure}

Compared with the case $m>2$, another fundamental difference 
when $m\in (1,2)$ is the fact that the energy corresponding
to the compactly supported steady state can be positive when $\nu$ is close to $\nu_\infty$.
Figure~\ref{fig:engbaren}(a) shows the evolution of the energy corresponding to the two dynamic simulations from Figure~\ref{fig:smallmevol}. 
The solution with small initial width converges to a {\em positive} energy,
while the solution with large initial width that decays to zero 
has energy that goes to zero as well. In particular, at $\nu=\nu_\infty$
the steady state is governed by~\eqref{eq:nleig} and $E[\rho_\infty] = 
(\frac{1}{m}-\frac{1}{2})\nu_\infty\int (\rho_\infty)^{m} >0$.
This observation once again confirms that for $\nu$ near the critical diffusion $\nu_\infty$, the compactly supported steady states can only be
{\em local} minimizers. It is possible that they turn into global minimizers once $\nu$ drops below a certain threshold (strictly smaller than $\nu_\infty$), but this aspect will be investigated elsewhere.

When the solution spreads and decays to zero, the dynamics is 
expected to be dominated by the nonlinear diffusion~\cite{Bed2010}.
For the decaying solution from Figure~\ref{fig:smallmevol}(b),
the rescaled profiles $\tilde{\rho}(\tilde{x},t) = 
\lambda^{-1}\rho(\lambda^{-1}\tilde{x},t)$ with $\lambda =
\max_x \rho(x,t)$ shown in Figure~\ref{fig:engbaren}(b) converge indeed to the rescaled Barenblatt profile, the solution of the porous medium equation $\rho_t = \nu \partial_x (\rho \partial_x \rho^{m-1})$.
The convergence however seems slow and is established 
only for very large time.


\appendix

\section{Proof of the last fact about equilibria in Section \ref{sect:statsol}}
\begin{lemma}
\label{lem:steady_sym}
Let $\rho$ be a stationary solution of \eqref{eq:main_evo} in one dimension, with compact support. Then there exists a symmetric stationary solution $\rhot$ such that
\begin{equation*}
    E[\rhot] = E[\rho].
\end{equation*}
\end{lemma}
\proof
Let $\mathrm{supp}[\rho]=[a,b]$ for some $a,b\in\R$. For a given $x\in(a,b)$ we have
\begin{equation}\label{eq:stat_rho}
    \e f(\rho(x)) = G*\rho ( x) + C
\end{equation}
for some $C\in \R$. Evaluation on $x=a$ and $x=b$ gives
\begin{equation*}
    C=-\int_a^b G(a-y)\rho(y) dy = -\int_a^b G(b-y)\rho(y) dy.
\end{equation*}
Let $\rhob(x)=\rho(x+x_0)$ with $x_0=(a+b)/2$. Then $\rhob$ is still a steady state and it satisfies $E[\rhob]=E[\rho]$ due to translation invariance. Moreover, the support of $\rhob$ is symmetric. Let us introduce
\begin{equation*}
    \rhot(x):=f^{-1}\left( \frac{1}{2}(f(\rhob)(x) + f(\rhob)(-x)) \right) .
\end{equation*}
Clearly, $\mathrm{supp}[\rhot]=\mathrm{supp}[\rhob]$ and we have, for all $x\in\mathrm{supp}[\rhot]$,
\begin{align*}
    & \e f(\rhot(x)) =\frac{\e}{2}(f(\rhob(x)) + f(\rhob(-x))) = \frac{\e}{2}(f(\rho(x+x_0)) + f(\rho(-x + x_0)))\\
    & \ = \frac{1}{2}\int_{a}^{b}G(x+x_0-y)\rho(y) dy + \frac{1}{2}\int_{a}^{b}G(-x+x_0-y)\rho(y) dy + C \\
    & \ = \frac{1}{2}\int_{(a-b)/2}^{(b-a)/2}G(x-z)\rhob(z) dy + \frac{1}{2}\int_{(a-b)/2}^{(b-a)/2}G(-x-z)\rhob(z) dy + C\\
    & \ = \frac{1}{2}\int_{(a-b)/2}^{(b-a)/2}G(x-z)\rhob(z) dy + \frac{1}{2}\int_{(a-b)/2}^{(b-a)/2}G(x-z)\rhob(-z) dy + C\\
    & \ = \int_{(a-b)/2}^{(b-a)/2}G(x-z)\frac{1}{2}\left(\rhob(z) + \rhob(-z)\right) dz + C = \int_{(a-b)/2}^{(b-a)/2}G(x-z)\rhot(z) dz + C
\end{align*}
where we have used the symmetry of $G$. The above computation shows that $\rhot$ has the same energy as $\rho$ .
\endproof


\end{document}